\documentclass[11pt,a4paper]{amsart}
\usepackage{amssymb,hyperref}
\usepackage[top=29mm,right=30mm,bottom=30mm,left=30mm]{geometry}

\newcommand{\AGL}{\mathop{\mathrm{AGL}}}

\newcommand{\Sp}{\mathop{\mathrm{Sp}}}
\newcommand{\SO}{\mathop{\mathrm{SO}}}
\newcommand{\SL}{\mathop{\mathrm{SL}}}
\newcommand{\PGL}{\mathop{\mathrm{PGL}}}

\newcommand{\GO}{\mathop{\mathrm{GO}}}
\newcommand{\PSL}{\mathrm{\mathrm{PSL}}}
\newcommand{\GU}{\mathop{\mathrm{GU}}}
\newcommand{\GL}{\mathop{\mathrm{GL}}}

\newcommand{\GSp}{\mathop{\mathrm{GSp}}}

\newcommand{\CSp}{\mathop{\mathrm{CSp}}}

\newcommand{\POmega}{\mathop{\mathrm{P}\Omega}}

\newcommand{\Aut}{\mathop{\mathrm{Aut}}}
\newcommand{\Out}{\mathop{\mathrm{Out}}}
\newcommand{\lcm}{\mathop{\mathrm{lcm}}}

\newcommand{\GammaL}{\mathop{\Gamma\mathrm{L}}}
\newcommand{\Alt}{\mathop{\mathrm{Alt}}}

\newcommand{\Sym}{\mathop{\mathrm{Sym}}}
\newcommand{\meo}{\mathop{\mathrm{meo}}}

\newcommand{\C}{\mathcal{C}}
\renewcommand{\wr}{\mathop{\mathrm{wr}}}
\def\pl{\oplus} 
\newtheorem{theorem}{Theorem}[section]

\def\la{\langle}
\def\ra{\rangle}
\newtheorem{corollary}[theorem]{Corollary}

\newtheorem{lemma}[theorem]{Lemma}
\newtheorem{remark}[theorem]{Remark}
\newtheorem{notation}[theorem]{Notation}

\newcounter{table2counter}
\newcommand{\tabtwolabel}[1]{\refstepcounter{table2counter} \label{#1}}
\newcounter{table3counter}
\newcommand{\tabthreelabel}[1]{\refstepcounter{table3counter} \label{#1}}
\newcounter{table4counter}
\newcommand{\tabfourlabel}[1]{\refstepcounter{table4counter} \label{#1}}
\newcounter{table5counter}
\newcommand{\tabfivelabel}[1]{\refstepcounter{table5counter} \label{#1}}
\newcounter{table6counter}
\newcommand{\tabsixlabel}[1]{\refstepcounter{table6counter} \label{#1}}
\newcounter{table7counter}
\newcommand{\tabsevenlabel}[1]{\refstepcounter{table7counter} \label{#1}}
\begin{document}

\title[Maximal element order]{Affine transformations of finite vector spaces with large orders or few cycles}

\author[S.~Guest]{Simon Guest}
\address{Simon Guest, Department of Mathematics, University of Southampton, Highfield, Southampton, SO17 1BJ. United Kingdom}\email{s.d.guest@soton.ac.uk}

\author[J. Morris]{Joy Morris}
\address{Joy Morris, Department of Mathematics and Computer Science,
University of Lethbridge,\newline
Lethbridge, AB. T1K 3M4. Canada}
\email{joy@cs.uleth.ca}

\author[C. E. Praeger]{Cheryl E. Praeger}
\address{Cheryl E. Praeger, Centre for Mathematics of Symmetry and Computation,
School of Mathematics and Statistics,
The University of Western Australia,\newline
 Crawley, WA 6009, Australia}\email{Cheryl.Praeger@uwa.edu.au}

\author[P. Spiga]{Pablo Spiga}
\address{Pablo Spiga,
Dipartimento di Matematica e Applicazioni, University of Milano-Bicocca,\newline
Via Cozzi 53, 20125 Milano, Italy}\email{pablo.spiga@unimib.it}

\thanks{Address correspondence to P. Spiga,
E-mail: pablo.spiga@unimib.it\\
The second author is supported in part by the National Science
 and Engineering Research Council of Canada.
This work forms part of Australian Research Council Project DP130100106 of the third author.}


\begin{abstract}
Let $V$ be a $d$-dimensional vector space over a field of prime order $p$. We classify the 
affine transformations of $V$ of order at least $p^{d}/4$, and 
apply this classification to determine the finite primitive permutation groups of affine 
type, and of degree $n$, that contain a permutation of order at least $n/4$. Using this result we obtain 
a classification of finite primitive permutation groups of affine type containing a permutation with at 
most four cycles. 
 \end{abstract}
\maketitle

\section{Introduction}\label{sec:intro}
That permutations of a set of size $n$ can have order as great as 
$e^{(1+o(1))(n\log n)^{1/2}}$ was shown by Edmund Landau~\cite{La1,La2} 
in 1903. However many of these large ordered permutations do not belong to 
proper primitive subgroups of $\Sym(n)$ or $\Alt(n)$. Indeed, in~\cite{GMPSorders} 
it was shown that the primitive permutation groups on $n$ points having a nonabelian socle, and containing a permutation of order at least $n/4$, are very restricted, with the natural actions of alternating groups $\Alt(r)$ on subsets, and projective groups $\PSL_r(q)$ on points or hyperplanes playing a special role: the socle of each such group is $\Alt(r)^\ell$ or $\PSL_r(q)^\ell$ acting on $\ell$-tuples of subsets, points or hyperplanes.
The case of primitive groups with an abelian socle was not treated in~\cite{GMPSorders}. These primitive groups are groups of affine transformations of finite vector spaces, where the point set is the vector space itself. 

The first aim of this paper is to determine the affine transformations of a vector space of size $n$ which have order at least $n/4$, and the affine primitive groups in which they lie.
Each affine transformation $g$ of a finite vector space $V$ has the form 
$g=t_vh$, with $t_v:x\mapsto x+v$ a translation, for
some $v\in V$, and with $h\in \GL(V)$, where $t_v$ is performed first followed by $h$.

\begin{theorem} \label{thm:orders}
Let $V$ be a $d$-dimensional vector space over a field $ \mathbb{F}_{p}$ of prime order $p$,
and let $g=t_vh$ be an affine transformation of $V$ with order at least $p^d/4$. 
Then $g$ and $h$ appear in one of the Tables~$\ref{tab:neworders1},~\ref{tab:neworders2},~\ref{tab:neworders3}$. 
\end{theorem}

\begin{remark} 
 \emph{We note that each of the examples $g$ in 
Tables~$\ref{tab:neworders1},~\ref{tab:neworders2},~\ref{tab:neworders3}$ have 
order at least $p^{d}/4$ when $p$ is odd. This is clear from the expressions for the 
orders of the elements in the tables. However there are a few instances where this is not the case when $p=2$, 
so Theorem~\ref{thm:orders} is not in fact an `if and only if' statement for $p=2$. For example in 
line~\ref{line14T3} of Table~\ref{tab:neworders2}, when $(d, d_1, d_2, d_3, d_4)=(12, 2, 2, 3, 5)$, 
the element order is $|g | = (2^2-1)(2^3-1)(2^5-1)= 651 < 2^d/4 = 1024$. } 
\end{remark}

Theorem~\ref{thm:HA} in conjunction with the results in~\cite{GMPSorders} 
gives a complete classification of all finite primitive groups of degree $n$ 
containing elements of order at least $n/4$. 
The group of affine transformations of $V$ is denoted $\AGL(V)$ and is 
called the affine general linear group of $V$. It is a semidirect product 
$T\cdot\GL(V)$, where $T$ is the group of translations and $\GL(V)$ is the group
of invertible linear transformations of $V$. Finite primitive groups of \textit{affine type}
are the subgroups of $\AGL(V)$, for some $V$, of the form $G=T\cdot G_0$, where
$G_0=G\cap\GL(V)$ acts irreducibly on $V$. For $V=\mathbb{F}_p^d$ with $d=1$, each transitive subgroup of $\AGL(V)$ contains translations of order $p$, so all such groups are examples. Theorem~\ref{thm:HA} classifies the examples with $d\geq2$. 

\begin{theorem}\label{thm:HA}
Let $G=T\cdot G_0\leq\AGL(V)$ be an affine primitive group on $V$ of degree $p^d$,
where $p$ is prime and $d\geq 2$. If $G$ contains a permutation $g$ of order at least $p^d/4$, then 
one of the following holds for $G_0$:
 \begin{enumerate}
\item $\SL_{d/r}(p^r) \le G_0 \le {\rm \Gamma L}_{d/r}(p^r)$ for some $r \mid d$ with $1\leq r < d$;
\item $G_0 \le {\rm \Gamma L}_1(p^d)$ with $[\GL_1(p^d) : G_0 \cap \GL_1(p^d)]\leq 3$;
\item $G_0 \le \GL_{d/r}(p)\wr \Sym(r)$ for some $r\mid d$ with $1<r\leq d$, and additionally, one of: 
\begin{itemize}
\item[$(i)$] $p=2$, and $d=r\leq 5$, or $d=2r\leq 6$, or $d\geq 3r$ and $4r^2-21r\leq d$,
\item[$(ii)$] $p=3$, and $d=r\leq 3$, or $r=2$,
\item[$(iii)$]$p\geq 5$ and $d=r=2$;
\end{itemize}
\item $p=2$ and $d \le 6$, or $p=3$ and $d \le 4$, or $d=2$ and $p \le 13$, and $G_0$ is in Table~$\ref{tab:HAgroups}$.
 \end{enumerate}
 \end{theorem}

 \begin{remark}
 \emph{ We note that in case (3) the image of $g$ in $\Sym(r)$ is trivial in most cases. We prove this in Lemma~\ref{lem:onlyC2s} where we also find the exact values where $g$ has a possibly non-trivial image in $\Sym(r)$. Case~(3)~(i) does not always occur; a necessary condition for existence is for such a group to contain an element in one of the lines 7--18 of Table~\ref{tab:neworders2}, or $d\le 5$. The function $4r^2-21r$ in case (3)~(i) is not the best possible lower bound of $d$ and for a refined version we refer to Remark~\ref{rem2}.
} 
 \end{remark}

Since the order of a permutation is equal to the least common multiple of the cycle lengths 
in its disjoint cycle representation, permutations with a bounded number of 
cycles have orders which grow at least linearly with the degree $n$: if a permutation 
has $c$ cycles, then one of its cycles has length at least $n/c$, and hence its order 
is at least $n/c$. Of course the converse is not true: permutations with order at least 
$n/4$ can have as many as $3n/4$ cycles of length 1. We apply our classification of affine 
transformations of order at least $n/4$ to determine all affine transformations which have 
at most 4 cycles, as well as the affine primitive groups which contain such elements.

\begin{theorem}\label{thm:element4cycles}
Let $V$ be a $d$-dimensional vector space over a field $\mathbb{F}_p$ of prime order $p$,
and let $g=t_vh$ be an affine transformation of $V$ with at most four cycles in its action on $V$. 
Then $g$ appears in one of the Tables~$\ref{tab:new1}, ~\ref{tab:new2},~\ref{tab:new3}$. 
\end{theorem}

 \begin{theorem}\label{thm:HA4cycles}
Let $G=T\cdot G_0\leq\AGL(V)$ be an affine primitive group on $V$ of degree $p^d$,
where $p$ is prime and $d\geq 2$. If $G$ contains a permutation with at most four cycles, 
then one of the following holds:
\begin{enumerate}
\item $\SL_{d/r}(p^r)\leq G_0\leq \GammaL_{d/r}(p^r)$ where $r$ divides $d$ and $1 \le r \le d$. Moreover, $G_0$ contains $s_{d/r}^i$, where $1\leq i\leq 3$ and $s_{d/r}$ denotes a Singer cycle in $\GL_{d/r}(p^r)$;
\item $G_0\leq \GL_{d/r}\wr\Sym(r)$ for some $r\mid d$ with $r>1$ and
\begin{itemize} 
\item[$(i)$] $p=2$ and $r=2$, or $d=r\leq 5$, or $(d,r)=(6,3)$,
\item[$(ii)$] $p=3$ and $d=r\leq 3$,
\item[$(iii)$]$p\geq 5$ and $d=r=2$;
\end{itemize}
\item $G_0$ is contained in one of the rows of Table~$\ref{tab:HAgroups}$ with a `y' in the fourth column.
 \end{enumerate}
 \end{theorem}

Bamberg and Penttila~\cite{BP} have obtained a very detailed classification of the groups satisfying part~(1). The classification in Theorem~\ref{thm:HA4cycles} could be refined taking into account the results of~\cite{BP}.

In~\cite{GMPScycles}, we build on these 
results to classify all finite primitive groups containing elements with at most 
four cycles. These results have various applications; in particular to normal 
coverings of a group and to the study of monodromy groups of Siegel functions. 
We refer the reader to~\cite{GMPScycles} for more details and also to~\cite{mueller}, 
where the finite primitive groups that contain a permutation with at most two cycles 
are classified.

The choice of ``$p^d/4$'' in Theorems~\ref{thm:orders} and~\ref{thm:HA} and of ``four'' in Theorems~\ref{thm:element4cycles} and~\ref{thm:HA4cycles} is to some extent arbitrary. On the one hand it allows a list of exceptions that is not too cumbersome to use and, on the other hand, it will be strong enough to determine in the forthcoming paper~\cite{chore} the first sharp bound on the normal covering number of $\Sym(n)$.

\begin{table}[htdp]
\begin{center}
\begin{tabular}{cccc}
\hline
$d$ & $p$ & $G_0$ & $G$ contains a permutation with \\
 & & & at most four cycles?\\\hline 
 4 & 2 & $\Alt(5)$ or $\Sym(5)$ & y \\ 
 4 & 2 & $\Alt(6)$ or $\Sym(6) \cong \Sp_4(2)$ & y \\ 
 4 & 2 & $\Alt(7)$ & y\\ 
 6 & 2 & $\Sp_6(2)$ & y \\ 
 6 & 2 & $\Sym(8) \cong \GO_6^{+}(2)$ or $\GO_6^{-}(2)$ & y\\ 
 6 & 2 & $\Sym(7)$ & y \\ 
 6 & 2 & $3 \times \GL_3(2)$ & y \\ 
 6 & 2 & $\Sym(3) \times \GL_3(2)$ & y \\ 
 3 & 3 & $\Omega_3(3)\cong \Alt(4)$ or $\Omega_3(3)\times 2\cong \Alt(4)\times 2$ & y\\ 
 3 & 3 & $\SO_3(3)\cong \Sym(4) \quad \text{(two such groups)}$ & y\\ 
 3 & 3 & $\GO_3(3)\cong \Sym(4)\times 2$ & y \\
 4 & 3 & $8.\Alt(5)$ or $8.\Sym(5)$ & y \\ 
 4 & 3 & $\CSp_4(3)$ & n \\ 
 4 & 3 & $\GO^{+}_4(3)$ or $\GO^{-}_4(3)$ & n \\ 
 4 & 3 & $\GL_2(3):(3\times \Sym(3))$ & n\\ 
 4 & 3 & $((2\times Q_8):2):5:4$ & n \\ 
 4 & 3 & $\GL_2(3):\Sym(4)$ & n \\ 
 4 & 3 & $2.\PGL_2(9)$ & n \\
 4 & 3 & $2^{1+4}.\Alt(5):2$ & n \\
 4 & 3 & $(\mathrm{SA}_{16}:2):3$ & n \\ 
 4 & 3 & $(\mathrm{SA}_{16}:2):6$ (two such groups) & n \\ 
 4 & 3 & $Q_8.\Sym(3):4$ & n \\ 
 4 & 3 & $\GL(2, 3):D_8$ & n \\ 
 2& 5 & $\SL_2(3):2$ & y \\ 
2& 5 & $\SL_2(3):4$ & y \\ 
2& 5 & $\Sym(3)$&y\\
2& 5 & $D_{12}$ &y\\ 
 2 & 7 & $\SL_2(3)$ & y \\
 2 & 7 & $3 \times \SL_2(3)$ & y \\
 2 & 7 &$3 \times \SL_2(3).2 = 3 \times 2.\Sym(4)$ &y \\ 
2&7&$D_{16}$& y\\
 2& 11 &$5 \times \GL_2(3)$ & y \\ 
 2& 11 & $5 \times \SL_2(5)$ & y \\ 
 2& 13 & $\SL_2(3):4$ or $3 \times \SL_2(3):4$& y \\ 
 \hline
 \end{tabular}
\caption{Primitive groups in Theorems~\ref{thm:HA}(4) and~\ref{thm:HA4cycles}(3)} \label{tab:HAgroups}.
\end{center}
\end{table}

\begin{table}[!h]
\begin{center}
\begin{tabular}{cccccllc}
\hline
Line & $d$ & $p$ & Conditions & $h$ & $|g|$ & $|g|$ vs. $|h|$ \\ \hline
 1 & $\ge 1$ & -- & $1\leq i\leq 3$ & $s_{d}^i$ & $(p^d-1)/i$&$|g|=|h|$\tabtwolabel{l1T2}\\
 & & & and $i\mid p^d-1$ & & \\
 2 & $\ge2$ & -- & $1\leq i \leq 3$ & $J_1\pl s_{d-1}^i$ & $(p^{d}-p)/i$ & $|g|=p|h|$ \tabtwolabel{l2T2}\\ 
 & & & and $i\mid p^{d-1}-1$ & & & \\
 3 & $1$ & -- & & $J_1$ & $p$ &$|g|=p|h|$\tabtwolabel{l3T2}\\
\hline
\end{tabular}
\end{center}
\caption{Arbitrary $p$}\label{tab:neworders1}
\end{table}

\begin{table}[!h] 
\begin{center}
\begin{tabular}{cccccllc}
 \hline
 Line & $p$ & $d$ & Conditions & $h$ & $|g|$ & $|g|$ vs. $|h|$ \\\hline
 1 & $3$ & $\ge 3$ & & $J_1\pl s_{d-1}$ & $3^{d-1}-1$&$|g|=|h|$ \tabthreelabel{line1T3} \\
 2 & $3$ & $\ge 3$ & $t=2$, $(d_1,d_2)=1$ & $s_{d_1}\pl s_{d_2}$ & $\frac{1}{2}(3^{d_1}-1)(3^{d_2}-1)$&$|g|=|h|$ \tabthreelabel{line2T3} \\
 3 & $3$ & $\ge 4$ & $i=1,2$ & $(s_{1}\otimes J_2)\pl s_{d-2}^{i}$ & $3(3^{d-2}-1)$ &$|g|=|h|$ \tabthreelabel{line3T3} \\
 & & & $d$ odd if $i=2$ & & \\
 4 & $3$ & $\ge 4$ & & $J_2\pl s_{d-2}$ & $3(3^{d-2}-1)$ &$|g|=|h|$ \tabthreelabel{line6T3}\\
 5 & $3$ & $\ge 4$ & $h' \in \GL_{d-1}(p)$ as $h$ in & $J_1\pl h'$ & $3|h'|$ &$|g|=p|h|$ \tabthreelabel{line4T3} \\
 & & & lines~\ref{line1T3}, \ref{line2T3} & & & \\
 6 & $3$ & $\ge 5$ & & $J_3\pl s_{d-3}$ & $9(3^{d-3}-1)$ &$|g|=p|h|$ \tabthreelabel{line5T3}\\\hline

 7 & $2$ & $\ge 3$ & $t \ge 2$, $d_1=1$ with & $ \bigoplus_j s_{d_j}$ & $\prod (2^{d_j}-1)$ &$|g|=|h|$ \tabthreelabel{line8T3} \\
 & & & $d_2,\ldots,d_t \ge 2$ coprime & & & \\
 8 & $2$ & $\ge 5$ & $t\ge 2$, with & $ \bigoplus_j s_{d_j}$ & $\prod (2^{d_j}-1)$ & $|g|=|h|$ \tabthreelabel{line9T3} \\
 & & & $d_j\ge 2$ coprime & & \\
 9 & $2$ & $\ge 4$ & $d_1=4$, and for $j\ge 2$ & $(s_2\otimes J_2 )\pl \bigoplus_{j\ge 2} s_{d_j}$ & $6 \prod (2^{d_j}-1)$ & $|g|=|h|$ \tabthreelabel{line10T3} \\
 & & &$d_j\ge 3$ odd and coprime & & \\
 10 & $2$ & $\ge 5$ & $d_1=3$, and for $j \ge 2$ & $J_3 \pl\bigoplus_{j\ge 2} s_{d_j}$ & $4 \prod (2^{d_j}-1)$ & $|g|=|h|$ \tabthreelabel{line11T3} \\
 & & & $d_j\ge 2$ and coprime & & &\\
 11 & 2&$\ge 4$ & $t\ge 2$, and for $j \ge 2$ & $J_2 \pl\bigoplus_{j\ge 2} s_{d_j}$ & $2 \prod (2^{d_j}-1)$ & $|g|=|h|$ \tabthreelabel{line12T3} \\
 & & & $d_j\ge 2$ and coprime & & &\\
 12 & $2$ & $\ge 5$ & $t\ge 2$, $d_j \ge 2$ and coprime & $ \bigoplus_j s_{d_j}$ & $\frac{1}{3}\prod (2^{d_j}-1)$ & $|g|=|h|$ \tabthreelabel{line14T3} \\
 & & & except $(d_1,d_2)=2$ & & & \\
 $ 13$ & $2$ & $\ge 5$ & $t\ge 2$, $d_1\geq 4$ even, & $s_{d_1}^3 \pl \bigoplus_{j\geq2} s_{d_j}$ & $\frac{1}{3}\prod (2^{d_j}-1)$ & $|g|=|h|$ \tabthreelabel{line14primeT3} \\
 & & & $d_j \ge 2$ and coprime & & & \\
 14 & $2$ & $\ge 4$ & $h' \in \GL_{d-1}(2)$ as $h$ & $J_1 \pl h'$ & $2|h'|$& $|g|=p|h|$ \tabthreelabel{line15T3} \\
 & & & in lines~\ref{line8T3}, \ref{line9T3}, \ref{line14T3}, \ref{line14primeT3} & & &\\
 15 & $2$ & $\ge 5$ & $h' \in \GL_{d-2}(2)$ as $h$ & $J_2 \pl h'$ & $4|h'|$& $|g|=p|h|$ \tabthreelabel{line16T3} \\
 & & & in lines \ref{line8T3}, \ref{line9T3}, \ref{line14T3}, \ref{line14primeT3} & & \\
 16 & $2$ & $\ge 5$ & $h' \in \GL_{d-4}(2)$ as $h$ in & $J_4 \pl h'$ & $8|h'|$ & $|g|=p|h|$ \tabthreelabel{line17T3} \\
 & & & line \ref{line9T3} or in Table~\ref{tab:neworders1} line~\ref{line1T3} & & & \\
 17 & $2$ & $\ge 4$ & $i=1,3$ & $J_2\pl s_{d-2}^{i}$ & $4(2^{d-2}-1)/i$ & $|g|=p|h|$ \tabthreelabel{line18T3}\\
 \hline
\end{tabular}
\end{center}
\caption{Other infinite families $p=2$, $3$}\label{tab:neworders2}
\end{table}
\begin{table}[!h]
\begin{center}
\begin{tabular}{cccccllc}
 \hline
 Line & $p$ & $d$ & Conditions & $h$ & $|g|$ & $|g|$ vs. $|h|$ \\ \hline
 1 & $\ge 3$ & $2$ & $1\leq i \leq 3$ & $s_1^{i} \otimes J_2$ & $p(p-1)/i$ & $|g|=|h|$ \tabfourlabel{line1T4} \\ 
 & & & and $i\mid p-1$ & & &\\
 2 & $3$ & $4$ & & $s_{2}\otimes J_2$ & $24$ & $|g|=|h|$ \tabfourlabel{line2T4} \\
 3 & $2,3$ & $p$ & & $J_p$ & $p^2$ &$|g|=p|h|$ \tabfourlabel{line3T4}\\
 4 & $2$ & $2,3,4,5$ & & $J_d$ & $2,4,4,8$ &$|g|=|h|$ \tabfourlabel{line4T4}\\
 5 & $2$ & $3$ & & $J_2\pl J_1=J_2\pl s_1$ & $2$ &$|g|=|h|$ \tabfourlabel{line9T4}\\
 6 & $2$ & $4$ & & $J_3\pl J_1$ & $4$ &$|g|=|h|$ \tabfourlabel{line11T4}\\ 
 7 & $2$ & $4$ & & $J_4$ & $8$ & $|g|=p|h|$ \tabfourlabel{line5T4} \\
 8 & $2$ & $4$ & & $J_2 \pl J_2$ & $4$ &$|g|=p|h|$ \tabfourlabel{line6T4}\\
 9 & $2$ & $4$ & & $J_2 \pl J_1 \pl J_1$ & $4$ &$|g|=p|h|$ \tabfourlabel{line7T4}\\
 10 & $2$ & $3$ & & $J_2\pl J_1=J_2\pl s_1$ & $4$ &$|g|=p|h|$ \tabfourlabel{line8T4}\\
 11 & $2$ & $3$ & & $J_1 \pl J_1 \pl J_1$ & $2$ &$|g|=p|h|$ \tabfourlabel{line10T4}\\\hline
\end{tabular}
\end{center}
\caption{Sporadic cases}\label{tab:neworders3}
\end{table}

\begin{remark}\label{rem:tables}{\rm
\begin{enumerate}
\renewcommand{\theenumi}{\alph{enumi}}
\item The notation used in Tables~2--7 is explained in Notation~\ref{not}.

\item The elements in Table~\ref{tab:neworders3}, line \ref{line10T4}, with $h=J_1\oplus s_2^3$ also occur in Table~\ref{tab:neworders1}, line 2 with $(p,d,i)=(2,3,3)$.

\item The elements in Table~\ref{tab:neworders3}, line 8, with $h=J_2\oplus s_2^3$ also occur in Table~\ref{tab:neworders2}, line 18 with $(d,i)=(4,3)$.


\item Not all Singer cycles $s_a$ in $\GL_a(p)$ are conjugate. Thus a line containing Singer cycles or their powers may represent several conjugacy classes of examples.
\end{enumerate}
}
\end{remark}

\section{Notation and preliminary observations} \label{not:HA}
Given a positive integer $d$ and a prime $p$, we denote by $V$ the $d$-dimensional 
vector space of row vectors over the field $\mathbb{F}_p$ of size $p$, and choose a basis $\{e_1,\dots,e_d\}$. As in Section~\ref{sec:intro} we represent an element 
$g\in \AGL(V)$ uniquely as $g=t_vh$ with $t_v:x\mapsto x+v$ a translation, for
some $v\in V$, and with $h\in \GL(V)$ (where $t_v$ is performed first). We first seek conditions on $v$ and $h$ so 
that $g$ has order at least $n/4=p^{d}/4$ and then, using these results as a starting point, we find conditions so that $g$ has at most 4 cycles in its action on
$V$ (including the zero vector). We let $|g|$ denote the order of the element $g$.

We use the following notation and information.

\begin{notation}\label{not}{\rm
\begin{enumerate}
\renewcommand{\theenumi}{\alph{enumi}}
\item We denote by $I$ the identity element of $\GL(V)$. For each $r \ge 1$ and $h\in \GL(V)$ define $h(r)$ by
\begin{equation}\label{hk}
h(r)=I+h+\dots+h^{r-1}.
\end{equation}\label{ya}

\item For $v'\in V$ and $g=t_vh$, the $g$-cycle containing $v'$ consists precisely of the vectors 
\[\{v'h^{r-1} + vh(r) - v\mid r \ge 1 \}.\]\label{yo}

\item For $1\leq j\leq d$, let $s_j$ denote a generator of a Singer cycle in $\GL_{j}(p)$ (an element of order $p^j-1$), and let $J_j$ denote
the cyclic unipotent element of $\GL_{j}(p)$ acting on $\langle e_1,\dots,e_j\rangle$ sending
$e_i$ to $e_i+e_{i+1}$ for $i<j$ and fixing $e_j$. We suppress the parameter $p$ as it
will be clear from the context. For convenience we also let $J_0$ denote the identity on the zero vector space.

If we write, for example, $h=J_j\pl J_i$ we will mean that $h$ acts as $J_j$ on
$\langle e_1,\dots,e_j\rangle$, and as $J_i$ on $\langle e_{j+1},\dots,e_{j+i}\rangle$ in the sense of mapping
$e_{j+s}$ to $e_{j+s+1}$ for $1\leq s<i$ and fixing $e_{j+i}$.\label{yi}

\item In Table~\ref{tab:neworders2}, wherever the notation $d_j$ is used, we have $1 \le j \le t$, and $\sum_{j=1}^t d_j=d$.

\item Whenever $h_j \in \GL(V)$ is indecomposable (where $V$ has dimension $d_j$) and we write $h_j=s_ju_j(=u_js_j)$, this indicates the Jordan decomposition with $u_j$ unipotent and $s_j$ semisimple. Since $h_j$ is indecomposable, $V=\oplus_{i=1}^{m_j} W_i$ is a sum of $m_j$ pairwise isomorphic irreducible $\mathbb{F}_p\la s_j\ra$-submodules of dimension $d_j' = d_j/m_j$. If we have $h$ instead of $h_j$, we omit the subscript $j$ throughout this notation.
\end{enumerate}}
\end{notation}

We start by recalling~\cite[Lemma~$2.2$]{GMPSorders}. (Here $\log_p(x)$ denotes the logarithm of $x$ to the base $p$ and $\lceil x\rceil$ denotes the least integer $k$ satisfying 
$x\leq k$.)

\begin{lemma}\label{uni}
Let $u$ be a unipotent element of $\GL_d(p^{f})$ where $p$ is
prime and $f\geq 1$. Then $|u|\leq p^{\lceil \log_p(d)\rceil}$ and equality holds if and only if the Jordan decomposition of $u$ has a block of size $b$ such that $\lceil\log_p(d)\rceil=\lceil\log_p(b)\rceil$.
\end{lemma} 

If $g=t_vh\in \AGL_d(p)$, then $|g|$ is either $|h|$ or $p|h|$, and Lemma~\ref{lem:gorder} explains which one holds.

\begin{lemma}\label{lem:gorder}
Let $h\in \GL(V)$ of order $k$ and let $h(k)$ be as in {\rm(\ref{hk})}.
Then
\begin{itemize}
\item[(a)] $(vh(k))h=vh(k)$, for every $v\in V$;
\item[(b)] $g=t_vh$ has order $pk$ if and only if $vh(k)\ne 0$, and in this case $g^k=t_{vh(k)}$;
\item[(c)] the following are equivalent:
\begin{itemize}
 \item[(i)] there exists $v$ such that $t_vh$ has order $pk$;
 \item[(ii)] $h(k)\ne0$;
\item[(iii)] the minimal polynomial $m_h(x)$ of $h$ is of the form $(x-1)^{(k)_p}f(x)$ for
some polynomial $f(x)$ coprime to $x-1$, where $(k)_p$ denotes the $p$-part of $k$.
\end{itemize}
\end{itemize}
\end{lemma}

\begin{proof}
Observe that the element $h(k)$ defined in (\ref{hk}) can
also be written as $h(k)= I+h^{-1}+\dots +h^{-k+1}$, and that
$g^k=(t_vh)^k=t_u$, where $u=v+vh^{-1}+\dots +vh^{-k+1}=vh(k)$.
Also $h(k)(I-h)=I-h^{k}=0$ and so $u=vh(k)=vh(k)h
=(vh(k))h=uh$, proving (a).
Since $g^k=t_u$, $|g|=pk$ if and only if $u\ne0$, proving (b).

Now we prove part (c). Suppose that (i) holds and let $v\in V$ be such that $|t_vh|=pk$. Then by part (b), we have $vh(k)\neq 0$ and hence $h(k)\ne0$, so (ii) holds.
Next suppose that (ii) holds and let $v\in V$ be such that $vh(k)\neq 0$. Since $h(k)\neq 0$, $m_h(x)$ does not divide
$x^{k-1}+ x^{k-2} + \cdots +1$, but since $h$ has order $k$, $m_h(x)$ divides $x^k-1$.
Write
$k = (k)_pm$ and observe that
 \begin{eqnarray*}
x^k-1 &=& (x^m-1)^{(k)_p} = (x-1)^{(k)_p}(x^{m-1} + \cdots +x +1)^{(k)_p}\\
 &=& (x-1)^{(k)_p}\left
(\prod_{\lambda \ne 1, \lambda^m=1} (x-\lambda)
\right)^{(k)_p}.
 \end{eqnarray*}
Since $m_h(x)$ does not divide $(x^k-1)/(x-1)$, we see that $m_h(x) = (x-1)^{(k)_p}f(x)$, where $f(x)$ divides $(x^{m-1} + \cdots +x +1)^{(k)_p}$, and hence is coprime to $x-1$, proving (iii).

Finally suppose that (iii) holds with $m_h(x)=(x-1)^{(k)_p}f(x)$ for some polynomial $f(x)$ coprime to $x-1$.
Since, as we showed above, $x-1$ has multiplicity $(k)_p$ in $x^k-1$, the polynomial $m_h(x)$ does not divide
$\ell(x)=(x^k-1)/(x-1)$. Hence $\ell(h)$, which equals $h(k)$, is not the zero map, and so there exists $v \in V$ such that $v h(k) \ne 0$. By part (b), $t_vh$ has order $pk$, and (i) holds.
\end{proof}

We give a useful corollary for the case where $g=t_vh$ has order $p|h|$.

\begin{corollary}\label{cor:kp}
If $g = t_v h$ has order $p|h|=pk$, then $h$ is conjugate to $J_{(k)_p} \pl h'$
for some $h' \in \GL_{d-(k)_p}(p)$.
\end{corollary}

\begin{proof}
Suppose $|g|=p|h|$ and write $k=|h|$. By Lemma~\ref{lem:gorder}(c), $(x-1)^{(k)_p}$ divides $m_h(x)$.
It follows that there exists $v\in V$ such that the $\mathbb{F}_p\langle h\rangle$-submodule generated by $v$ is cyclic of dimension
$(k)_p$, and $h$ induces $J_{(k)_p}$ on it. By Lemma~\ref{uni}, we have $|J_{(k)_p}|=(k)_p$.
Since $(x-1)^{(k)_p+1}$ does not divide $m_h(x)$, the map $h$
does not involve $J_{(k)_p+1}$ by Lemma~\ref{uni}, and hence by~\cite[Theorem~8.2]{HH}, $h$ is conjugate to $ J_{(k)_p} \pl h'$
for some $h' \in \GL_{d-(k)_p}(p)$.
\end{proof}

\begin{lemma}\label{sub-restrict}
Let $g=t_vh\in\AGL(V)$ and let $U$ be the $(x-1)$-primary component of the $\mathbb{F}_p\langle h\rangle$-module $V$. Then $g$ is conjugate to $t_uh$ for some $u\in U$. In particular, if $h$ is fixed point free on $V \setminus \{0\}$, then $g$ is conjugate to $h\in\GL(V)$.
\end{lemma}

\begin{proof}
Let $V=U\oplus W$ be an $h$-invariant decomposition (so $W$ is the direct sum of the other primary components, if any).
Then $h|_W$ is fixed point free, so also $(h^{-1})|_W$ is fixed point free and in particular $(I-h^{-1})|_W$ is nonsingular. Observe that from Lemma~\ref{lem:gorder}(a) we have $vh(k)\in U$. Now
$v=u+w$, for some $u\in U$ and $w\in W$, and $vh(k)=(u+w)h(k)=uh(k)+wh(k)$ with
$uh(k)\in U$ and $wh(k)\in W$. Thus $wh(k)=vh(k)-uh(k)\in U\cap W$, so $wh(k)=0$. Since
$(I-h^{-1})|_W$ is nonsingular, there exists $w'\in W$ such that $w=w'-w'h^{-1}$, and
hence we have 
\[
t_{w'}^{-1}(t_vh)t_{w'}=t_{v-w'} (ht_{w'}h^{-1})h=t_{v-w'+w'h^{-1}}h=t_{v-w}h=t_{u}h.
\]
Note that if $h$ is fixed point free (that is, $U=0$), then we have $u=0$ and therefore $g$ is conjugate to $h\in\GL(V)$.
\end{proof}

\begin{enumerate}
\item[(f)] We add to our Notation~\ref{not} the subspaces $U$ and $W$ as defined in the statement and in the proof of Lemma~\ref{sub-restrict}, and define the integer $a$ by the equation
\[
 |U|=p^a.
\]
Since conjugate permutations have the same order and the same cycle structure we may, because of Lemma~\ref{sub-restrict}, assume from now on that $v\in U$.

\item[(g)]For a finite group $G$, we write $\meo(G)=\max\{|g|\mid g\in G\}$ for the maximal order of the elements of $G$. Given two natural numbers $n$ and $m$, we write $(n,m)$ for the greatest common divisor of $n$ and $m$.
\end{enumerate}
\section{Proof of Theorem~\ref{thm:orders}}

\begin{proof}[Proof of Theorem~$\ref{thm:orders}$~]
Suppose that $g=t_vh$ in $\AGL(V)=\AGL_d(p)$ has order $|g| \ge n/4$. We shall prove that $g$ and $h$ appear in one of the lines of Tables~\ref{tab:neworders1},~\ref{tab:neworders2} or~\ref{tab:neworders3}. Several times in the proof we use the facts that $\meo(\GL_{d}(p))=p^{d}-1$ and $\meo(\SL_{d}(p))=(p^{d}-1)/(p-1)$. A proof of these facts can be deduced, for example, from~\cite[Corollary~$2.7$]{GMPSorders} and for odd $p$ from~\cite[Table A.1]{KS}. 

\medskip\noindent
{\it The case $|g|=|h|$.}\quad 
First we assume that $|g|=|h|$.
We use the notation in Section~\ref{not:HA} and we suppose that $|h| \ge p^{d}/4$. 
Write $V=V_1\oplus V'$, where $V_1, V'$ are $h$-invariant and $h_1=h|_{V_1}$ is indecomposable; let $h'=h|_{V'}$ (possibly $V'=0$) so that $h=h_1\pl h'$, and let $k=\dim(V_1)$. 

\medskip\noindent
{\it Case $|g|=|h|$, $p\geq5$.}\quad We have $h^{p-1} = h_1^{p-1}\pl (h')^{p-1}$, and $h_1^{p-1}$, $(h')^{p-1}$ have determinant $1$. If $V' \ne 0$, then 
\[ 
|h^{p-1}| \le {\rm meo}({\rm SL}_k(p)) {\rm meo}({\rm SL}_{d-k}(p)) \le \frac{(p^k-1)(p^{d-k}-1)}{(p-1)^{2}} <\frac{p^{d}}{(p-1)^{2}}. 
\]
Hence $|h| < p^{d}/(p-1) \le p^{d}/4$, which is a contradiction since $p\ge5$. Hence $V'=0$ and $h=h_1$ is indecomposable. 


 If $h$ is irreducible then $h = s_d^{i}$ for some $i$ with $i\mid |s_d|$. As $|s_d|=p^d-1$, we have $1\leq i \leq 3$, as described in line~\ref{l1T2} of Table~\ref{tab:neworders1}. Suppose then that $h$ is not irreducible. Let $h=su$. Since $h$ is not irreducible, $m\geq 2$ and $u\ne1$. If $d'=1$, then by Lemma~\ref{uni} we get 
\[
|h| \le (p-1) p ^{\lceil \log_p(d) \rceil}.
\]
If $d \ge 6$, this gives no examples since
$(p-1) p ^{\lceil \log_p(d) \rceil} \le (p-1) dp < p^{d}/4.$ 
If $d\leq5$ then, since $p \ge 5$, $\lceil \log_p(d) \rceil=1$ and a direct calculation shows that $(p-1)p$ is less than $p^{d}/4$ unless $m=d=2$, which yields the example in line~\ref{line1T3} of Table~\ref{tab:neworders3}. 
Now assume that both $m,d' \ge 2$. Then $h$ has a conjugate lying in the subgroup $\GL_m(p^{d'})$. Under this conjugation, $s$ becomes a scalar matrix in $\GL_m(p^{d'})$ and $u$ becomes a unipotent element of $\GL_m(p^{d'})$. Applying Lemma~\ref{uni}, we have
\[
|h| \le (p^{d'}-1)p^{ \lceil \log_p(m) \rceil} 
\]
and $\lceil \log_p(m) \rceil \le m -1$ (see the proof of~\cite[Lemma~2.4]{GMPSorders}). Therefore, noting that $d'+m\leq d'm=d$ for integers $d',m\geq2$, we get
\[
|h| \le p^{d'} p^{m-1} = p^{d'+m-1} \le p^{d'm-1}< p^{d}/4 
\]
so there are no further examples when $p\ge5$.


\medskip\noindent
{\it Case $|g|=|h|$, $p=3$.}\quad
Arguing in the same way as in the first paragraph of ``Case $p\ge 5$'' we obtain that $V$ is not a direct sum of three non-zero $\mathbb{F}_p\la h\ra$-submodules.
First suppose that $h$ is indecomposable. If $h$ is semisimple then it is irreducible and so $h=s_d^i$ with $i\leq3$ as in line~\ref{l1T2} of Table~\ref{tab:neworders1}. Suppose now that $h$ is not semisimple and let $h=su$. Then $m \ge 2$ and
\[
|h| \le (3^{d'}-1)3^{\lceil \log_3(m) \rceil}.
\]
A direct calculation shows that this is less than $3^{d}/4$ unless $(d',m)=(1,2)$ or $(2,2)$; these cases yield the examples in lines~\ref{line1T4} and~\ref{line2T4} of Table~\ref{tab:neworders3}.

Finally suppose that $V=V_1\oplus V_2$ and $h=h_1\pl h_2$ with $h_i=h|_{V_i}$ indecomposable 
and $d_i=\dim(V_i)$ for $i=1,2$. If
 $h$ is semisimple then $h_1$ and $h_2$ are contained in Singer cycles. In this case if $(d_1,d_2) \ge 2$ then
 \[ 
 |h| \le {\rm lcm}\{3^{d_1}-1,3^{d_2}-1\} = (3^{d_1}-1)(3^{d_2}-1)/ (3^{(d_1,d_2)}-1) < 3^d/4.
 \]
 So $(d_1,d_2)=1$ and we have the examples in lines~\ref{line1T3} and~\ref{line2T3} of Table~\ref{tab:neworders2} (the condition $d\geq 3$ follows from a calculation). 
Now assume that $h$ is not semisimple. Then, replacing $h_1$ by $h_2$ if necessary, we may assume that $h_1=u_1s_1$ with $V_1$ and $m_1\ge 2$. Since $\meo(\GL_{d-d_1}(3))=3^{d-d_1}-1$ we have
 \[ 
 |h| \le |h_1| \cdot |h_2| \le (3^{d_1'}-1)3^{\lceil \log_3(m_1) \rceil} \meo(\mathrm{GL}_{d-d_1}(3))=(3^{d_1'}-1)3^{\lceil\log_3(m_1)\rceil}(3^{d-d_1}-1). 
 \]
Now a direct calculation shows that this is less than $3^d/4$ unless $(m_1,d_1')=(2,1),(2,2)$. In the second case, $h_1 \in \langle s_2 \otimes J_2 \rangle$ and so $|h^2| \le |h_1^2| |h_2^2| \le 12 \meo(\mathrm{SL}_{d-4}(3))$. But~\cite{KS} implies that $\meo(\SL_{d-4}(3))=(3^{d-4}-1)/2$, and hence $|h|\leq 12(3^{d-4}-1)$. However, since $d=4+d_2\geq5$, we have $12 (3^{d-4}-1) < 3^{d}/4$. Thus $(m_1,d_1')=(2,1)$ and $h_1 \in \langle s_1 \otimes J_2 \rangle$. If $h_2$ is semisimple then it is contained in a Singer cycle, giving the examples in lines~\ref{line3T3} (if $|h_1|=6$) and~\ref{line6T3} (if $|h_1|=3$) of Table~\ref{tab:neworders2} (the conditions $d\geq 4$ and $d$ odd, when $i=2$, follow from a calculation; note that $i=2, d=3$ gives $h=(s_1\otimes J_2)\oplus J_1$ of order $6 < 3^3/4$). 
 If $h_2$ is not semisimple, then $h_2=u_2s_2$ and $m_2 \ge 2$. We have
 \[ 
 |h| \le |h_1| \cdot |h_2| \le 6 (3^{d_2'}-1)3^{\lceil \log_3(m_2) \rceil} < 3^{d}/4 \]
 except for $d_2'=1$ and $m_2=2$. In this exceptional case, $|h_2|$ divides $6$ and, as $h_1\in\langle s_1\otimes J_2\rangle$, we have $|h|\leq 6<3^{d}/4$. Therefore there are no further examples when $p=3$.

\medskip\noindent
{\it Case $|g|=|h|, p=2$.}\quad
Suppose that $V=V_1\oplus\dots\oplus V_t$ and $h = h_1 \pl\cdots\pl h_t$, where the $V_i$ are $h$-invariant with $\dim(V_i)=d_i$, and each $h_i=h|_{V_i}$ is indecomposable. Also let $h_i=s_iu_i$.
 First suppose that $t=1$. If $h$ is semisimple then it is contained in a Singer cycle giving the examples in line~\ref{l1T2} of Table~\ref{tab:neworders1}. Suppose now that $u_1\ne1$, so that $m_1 \ge 2$. Then
 \[ 
 |h| \ \le (2^{d_1'}-1)2^{\lceil \log_2(m_1) \rceil} , 
 \]
 which is less than $2^d/4$ unless $(d_1',m_1)=(1,2),(1,3),(1,4),(1,5),(2,2)$; thus we have the examples in line~\ref{line4T4} of Table~\ref{tab:neworders3} and line~\ref{line10T3} of Table~\ref{tab:neworders2} (by taking $d=4$). 

 So we may assume $t \ge 2$. Observe that if $d=2$, then $h=h_1=h_2=1$ as in line~\ref{l1T2} of Table~\ref{tab:neworders1} (by taking $i=3$). Thus we may suppose that $d\geq 3$. If $h$ is semisimple, then each $h_i$ is irreducible 
and $|h| \le {\rm lcm} \{ 2^{d_i}-1 \mid i=1,\ldots,t\}$. If $(d_j,d_k) \ge 3$ for some distinct $j$ and $k$, then
 \[
 |h| \le {\rm lcm} \{ 2^{d_i}-1\mid i=1,\ldots,t \} \le \left(\prod_{i=1}^t (2^{d_i}-1)\right)/(2^{d_j}-1,2^{d_k}-1) < 2^{d}/7,
 \]
which is a contradiction. If there are at least three even $d_i$, then
 \[
 |h| \le {\rm lcm}\{ 2^{d_i}-1 \mid i=1,\ldots,t\} \le \left(\prod_{i=1}^t (2^{d_i}-1)\right)/(2^{2}-1)^{2} < 2^{d}/9,
 \]
again a contradiction. Moreover, as $d\geq 3$, if the fixed point subspace of $h$ has dimension at least $2$ (so at least two of the $d_i$ equal $1$), then $|h| \le\meo(\GL_{d-2}(2))= 2^{d-2}-1<2^{d}/4$. Thus $d\geq 3$, at most one $d_i$ can be $1$, at most 2 of the $d_i$ are even, and $(d_j,d_k)\leq 2$ for distinct $j,k$. Observe further that if $d_i=1$ and if $(d_j,d_k)=2$ for some distinct $j$ and $k$, then $|h|\leq (2^{d-1}-1)/3<2^{d-2}$: this shows that if $d_i=1$ for some $i$, then $(d_j,d_k)=1$ for distinct $j$ and $k$. The only such examples (with $t\ge 2$) are listed in lines~\ref{line8T3},~\ref{line9T3},~\ref{line14T3},~\ref{line14primeT3} of Table~\ref{tab:neworders2}. (Observe that in line~\ref{line14primeT3} we have $d_1\geq 4$ because $s_2^3=1$ fixes a subspace of dimension $2$.) 

Suppose now that $h$ is not semisimple. Then we may assume that $d_1$ is maximal such that $h_1=s_1u_1$ is non-semisimple, and that $m_1\geq2$. Now
 \[
 |h|\leq|h_1||h_2|\leq (2^{d_1'}-1)2^{\lceil\log_2(m_1)\rceil}\meo(\mathrm{GL}_{d-d_1}(2)) = (2^{d_1'}-1)2^{\lceil\log_2(m_1)\rceil}(2^{d-d_1}-1)
 \]
and a direct calculation shows that this is less than $2^d/4$ unless $(m_1,d_1')=(2,1)$, $(3,1)$ or $(2,2)$ (note that $(m_1,d_1')\in\{(2,1), (3,1), (4,1), (5,1), (2,2)\}$ if $d_1\leq 5$ from our work above). We consider each possibility for $(m_1,d_1')$.

 If $(m_1,d_1') =(2,2)$, then $h_1=s_2\otimes J_2$, $|h_1|=6$ and $h=h_1 \pl h'$. Clearly $h'$ must have odd order for otherwise $|h| = \lcm \{6, |h'|\}\le 3 |h'| < 3 \cdot 2^{d-4}< 2^{d}/4$. It follows that $h'$ is semisimple with irreducible blocks of dimensions $d_2, \ldots, d_t$. Note that if one of the $d_i$ ($i\geq2$) is even, then $2^{d_i} \equiv 1 \pmod{3}$ and we have
 \[|h| \le {\rm lcm} \{6, 2^{d_j}-1\mid j=2,\ldots,t\} \le 6 \cdot 2^{d-4}/3 < 2^{d}/4.\] 
 Hence each $d_i$ is odd for $i\geq2$. If two of the $d_i$ have a common factor $>1$, then similarly, we have
 \[
 |h| \le {\rm lcm} \{6, 2^{d_i}-1\mid i=2,\ldots,t\} \le 6 \cdot 2^{d-4}/3 < 2^{d}/4.
 \]
Therefore the $d_i$ must be pairwise coprime. A similar calculation shows that none of the $d_i=1$, and that each of the $h_i$ ($i\geq2$) is $s_{d_i}$ (and not a proper power). Thus we have the examples in line~\ref{line10T3} of Table~\ref{tab:neworders2}. 

If $(m_1,d_1') =(3,1)$, then $h_1=J_3$ and $|h_1|=4$. As in the previous case, $h=h_1\pl h'$ and $h'$ must have odd order. So $h'$ is semisimple with irreducible blocks of dimensions $d_2, \ldots, d_t$, and the same arguments show that the $d_i$ are pairwise coprime. If each $d_i \ge 2$ then we have the examples in line~\ref{line11T3} of Table~\ref{tab:neworders2}. If some $d_i=1$ then $t=2$ and we have the example in line~\ref{line11T4} of Table~\ref{tab:neworders3}.

It remains to consider the case $(m_1,d_1')=(2,1)$, where $h_1=J_2$ of order $2$. As before, $h=h_1\pl h'$ where $h'$ is semisimple, and the usual arguments give us the examples in line~\ref{line12T3} of Table~\ref{tab:neworders2} and line~\ref{line9T4} of Table~\ref{tab:neworders3}. 

\medskip\noindent
{\it The case $|g|=p|h|$.}\quad
We have now classified the examples with $|g|=|h|$. Henceforth we assume that $|g|=p|h|$. By Lemma~\ref{lem:gorder}, the power $(x-1)^{(k)_p}$ divides $m_h(x)$, where $k=|h|$. If $d=1$, then $h=1$ and we have the examples in line~\ref{l3T2} of Table~\ref{tab:neworders1}. For the rest of the proof we assume then $d\geq 2$.

\medskip\noindent{\it Case $|g|=p|h|, p\geq5$.}\quad
First suppose that $h$ is semisimple. We seek conditions on $h$ for which $|h| \ge p^{d-1}/4$. In this case, $x-1 \mid m_h(x)$ and so 
we can write $h=J_1\pl h'$ where $h' \in {\rm GL}_{d-1}(p)$ is semisimple. (Recall that $d-1\geq 1$.) By our previous work, the only semisimple elements $h' \in {\rm GL}_{d-1}(p)$ of order at least $p^{d-1}/4$ appear in line~\ref{l1T2} of Table~\ref{tab:neworders1}; 
thus the only examples of $h$ occur in line~\ref{l2T2} of Table~\ref{tab:neworders1}.
If $h$ is not semisimple then by Corollary~\ref{cor:kp} we have $h=J_{(k)_p}\pl h'$, and $(k)_p \ge 5$. In particular,
\[|h| \le (k)_p \meo({\rm GL}_{d-(k)_p}(p)) = (k)_p(p^{d-(k)_p}-1)<p^{d-1}/4 \]
in all cases since $(k)_p\ge 5$. So there are no further examples when $p \ge 5$.


\medskip\noindent
{\it Case $|g|=p|h|, p=3$.}\quad 
 If $h$ is semisimple then by Corollary~\ref{cor:kp} we can write $h=J_1\pl h'$ with $h' \in \GL_{d-1}(3)$ of order at least $3^{d-1}/4$. Therefore $h'$ is contained in line~\ref{l1T2} of Table~\ref{tab:neworders1} or lines~\ref{line1T3},~\ref{line2T3} of Table~\ref{tab:neworders2}; 
 so $h$ is as in line~\ref{l2T2} of Table~\ref{tab:neworders1} or line~\ref{line4T3} of Table~\ref{tab:neworders2}. 
 If $h$ is not semisimple then, by Corollary~\ref{cor:kp}, $h$ is of the form $J_{(k)_3} \pl h'$ (with $k\geq 3$). If $h=J_{(k)_3}$, then $|h|=(k)_3=d$ so $|g|=3d$ and $3^d/4 >3d$ for $d \ge 5$, so $h=J_3$ (line~\ref{line3T4} of Table~\ref{tab:neworders3}). 
Otherwise, $d >(k)_3$
 and
 \[ 
|h| \le (k)_3 (3^{d-(k)_3}-1),
\]
 which is less than $3^{d-1}/4$ unless $(k)_3=3$ and $d \ge 5$. So we may assume that $h=J_3\pl h'$ where $h' \in \GL_{d-3}(3)$. Now if $3$ divides $|h'|$ then $|h|=|h'| \leq \meo(\GL_{d-3}(3))=3^{d-3}-1<3^{d-1}/4$, which is a contradiction. So $h'$ is semisimple and therefore appears in line~\ref{l1T2} of Table~\ref{tab:neworders1} or lines~\ref{line1T3},~\ref{line2T3} of Table~\ref{tab:neworders2}. 
 However it is clear that $|h|<3^{d-1}/4$ if $h'$ is as in lines~\ref{line1T3},~\ref{line2T3} of Table~\ref{tab:neworders2}: so the only additional examples are in line~\ref{line5T3} of Table~\ref{tab:neworders2}. 


\medskip\noindent
{\it Case $|g|=p|h|, p=2$.}\quad Again we first suppose that $h$ is semisimple so that $h=J_1 \pl h'$ and $|h'| \ge 2^{d-1}/4$; that is, $h'$ is one of the semisimple examples in line~\ref{l1T2} of Table~\ref{tab:neworders1} or lines~\ref{line8T3},~\ref{line9T3},~\ref{line14T3} or~\ref{line14primeT3} of Table~\ref{tab:neworders2}; 
thus the only such examples are in line~\ref{l2T2} of Table~\ref{tab:neworders1}, or in line~\ref{line10T4} of Table~\ref{tab:neworders3} (arising from $h'$ as in line~\ref{l1T2} of Table~\ref{tab:neworders1} with $d=2$ and $i=3$), or in line~\ref{line15T3} of Table~\ref{tab:neworders2}. 
Now suppose that $h$ is not semisimple; so by Corollary~\ref{cor:kp}, $h=J_{(k)_2}\pl h'$ (with $k\geq 2$). If $h=J_{(k)_2}$, then $|h|=(k)_2=d$ so $|g|=2d$ and $2^d/4 >2d$ for $d \ge 6$, so $h=J_2$ or $J_4$ (lines~\ref{line3T4} and~\ref{line5T4} of Table~\ref{tab:neworders3}). Otherwise, $d > (k)_2$
 and
\[ 
|h| \le (k)_2 (2^{d-(k)_2}-1),
\]
which is less than $2^{d-1}/4$ unless $(k)_2=2$ or $4$. Suppose first that
$(k)_2=2$. Then 
$h=J_2\pl h'$, with $h'\in \GL_{d-2}(2)$. If $h'$ is semisimple then $|h|=2|h'|$ and $|h'|\geq 2^{d-2}/4$, so as in the previous paragraph, 
$h'$ is a semisimple element as in line~\ref{l1T2} of Table~\ref{tab:neworders1}, or lines~\ref{line8T3},~\ref{line9T3},~\ref{line14T3} or~\ref{line14primeT3} of Table~\ref{tab:neworders2}: 
 hence $h$ is one of the examples in lines~\ref{line16T3},~\ref{line18T3} of Table~\ref{tab:neworders2} or lines~\ref{line7T4},~\ref{line8T4} of Table~\ref{tab:neworders3}. 
 If $h'$ is not semisimple then $|h|=|h'|$, which is at least $2^{d-1}/4$ if and only if $h' \in \GL_{d-2}(2)$ is a non semisimple element of this order; by our previous work, this only occurs if $d=4$ and $h'=J_2$ (note that $h'$ cannot be $J_3$ from line~\ref{line4T4} of Table~\ref{tab:neworders3} since we are assuming $(k)_2=2$, but $k=|h|=4$ if $h=J_2\oplus J_3$); thus we have line~\ref{line6T4} of Table~\ref{tab:neworders3}. 
 Suppose now that $(k)_2=4$ so that $h=J_4\pl h'$. 
 If $h'$ is not semisimple then $|h|\le 2 |h'|$ and $2|h'|$ is at least $2^{d-1}/4$; this holds if and only if $h' \in \GL_{d-4}(2)$ is a non semisimple element of order at least $2^{d-3}$. There are therefore no such elements and we conclude that $h'$ is semisimple and $|h|=4|h'|$. Now $|h| \ge 2^{d-1}/4$ if and only if $|h'| \ge 2^{d-4}/2$ and the only such examples $h'$ occur in line 1 of Table~\ref{tab:neworders1} (with $i=1$ and $d-4 \ge 2$) and line~\ref{line9T3} of Table~\ref{tab:neworders2}; 
 thus we have the examples in line~\ref{line17T3} of Table~\ref{tab:neworders2}. 
\end{proof}

\section{Classification of elements with at most four cycles}

We now refine the list of affine transformations of order at least $n/4$ to determine those elements that have at most $4$ cycles in $V$. Recall the notation from Section~\ref{not:HA}, especially for $g=t_vh, V, U, W, p^a=|U|$, 
and assume that $g$ has at most four cycles in $V$. By Lemma~\ref{sub-restrict} we may assume that $v\in U$. We start with some further observations.

\subsection{$g$-invariant subsets of $V$}\label{sub-inv}
For each $h$-invariant subspace $V'$ of $V$, the subspace $U+V'$ is a $g$-invariant subset of
$V$ (recall that the subspace $U$ is defined in Notation~\ref{not}~(f)). In fact, for $u'+v'\in U+V'$, we have $(u'+v')g=(u'+v')t_vh=(v+u')h+v'h\in U+V'$.
In particular, taking $V'=0$, we see that $U$ is $g$-invariant.

\subsection{Three claims}\label{sub:4claims}

\medskip\noindent\emph{Claim~$1$:} Suppose that $V\ne U$, and let $W'$ be a
nontrivial $h$-invariant subspace of $W$.
Suppose that $g$ has $t$ cycles in $U$ and that $h$ has $r$ cycles in $W'$.
Then
\begin{enumerate}
 \item[(a)] $t\cdot r\leq 4$ with $t\geq1, r\geq2$;
\item[(b)] if $t\geq2$, then $t=r=2$, $W'=W$, and $h|_W$ is transitive on $W\setminus\{0\}$;
so $h=h|_U \pl s_{d-a}$.
\end{enumerate}

\begin{proof}[Proof of Claim $1$]Since $v\in U$, it follows that $U\oplus W'$ is
$g$-invariant. Let $w\in W'$ and $x\in U$. By Notation~\ref{not}~(b), the $g$-cycle containing $x+w$, where $x\in U, w\in W'$, consists precisely
of the vectors $xh^i+vh(i+1) - v + wh^{i}=x^{g^i}+w^{h^i}$, for $i\geq0$ (the equality can be easily proved by induction on $i$). This $g$-cycle is contained in
$x^{\langle g\rangle} + w^{\langle h\rangle}=\{x^{g^i}+w^{h^j}\,|\,\mbox{for all}\ i,j\}$.
It follows that there are at least $tr$ cycles of $g$ in $U\oplus W'$. Since $g$ has at most
4 cycles, this implies part (a), and, if $t\geq2$, then $t=r=2$, $W'=W$, and $\langle h\rangle$ is transitive 
on the non-zero vectors of $W'$.
\end{proof}

\medskip\noindent\emph{Claim $2$: } Let $|h|_U|=p^c$, $|g|=p^\delta k$ where $\delta=0, 1$, and
let $g$ have $t$ cycles in $U$. Then
\begin{equation}\label{orderg}
p^a\leq t|g|_U|= tp^{c+\delta}\leq \left\{ \begin{array}{ll}
 tp^{\delta}=t&\mbox{if}\ a=0,\ \mbox{that is, if $U=0$}\\
 tp^{\delta+\lceil \log_p(a)\rceil}&\mbox{if}\ a>0, \ \mbox{that is, if $U\ne0$}.\\
 \end{array}\right.
\end{equation}
The subspace $U$ is a single $g$-cycle if and only if either (i) $a=c=\delta=0$
and $h|_U=J_0$, or (ii) $\delta=1$,
$h|_U=J_a$ , and $a=1$ or $(a,p)=(2,2)$.

\begin{proof}[Proof of Claim $2$]
Since $g$ has $t\leq 4$ cycles in $U$, we have $p^a=|U|\leq t|g|_U|$. If $a=0$ then
$c=0$, $h|_U=J_0$, and $|g|_U|=p^\delta=1$. Thus if $a=0$ then the inequality~\eqref{orderg} holds, $U$ is a single $g$-cycle, and the conditions (i) hold. We may therefore assume that $a\geq1$. Then, by Lemma~\ref{uni},
$|g|_U|=p^\delta\,|h|_U|\leq p^{\delta+\lceil
\log_p(a)\rceil}$ with equality if
and only if $h|_U$ involves a cyclic matrix $J_b$ such that $\lceil \log_p(b)\rceil=\lceil
\log_p(a)\rceil$. In particular (\ref{orderg}) holds.

Suppose $U$ is a single $g$-cycle. Then (\ref{orderg}) holds with $t=1$, and hence $p^a\leq p^{\delta+\lceil\log_p(a)\rceil}$, that is, $a\leq \delta+\lceil \log_p(a)\rceil$. It follows from a computation, since $a\geq1$, that $\delta=1$ and either $1\leq a\leq 2$, or $(a, p)=(3, 2)$. If $a>1$, then from the inequalities $p^a\leq p^{c+1}\leq p^{1+\lceil \log_p(a)\rceil}$ in~\eqref{orderg}, we obtain $c=a-1$, that is, $|h|_U|=p^{a-1}$. Therefore,
by Corollary~\ref{cor:kp}, $h|_U$ involves $J_{(k)_p}$.
In particular if $(a,p)=(3,2)$ then $(k)_p\geq p^c=4$ but $h|_U$ does not involve $J_4$ because $U$ has dimension $3$ only.
Similarly if $a=2$ and $p$ is odd, then $(k)_p\geq p^c=p$, but $h|_U$
does not involve $J_p$ because $U$ has dimension $2$ only.
So $a=1$ or $(a,p)=(2,2)$, and in either case, $h|_U=J_a$ so part (ii) holds.

Conversely if $\delta=1$ and $h|_U=J_a$ with either $a=1$ or $(a,p)=(2,2)$, then $|h|_U|=p^{a-1}$
and so $|g|_U|=p^a=|U|$ so that $U$ must form a single $g$-cycle.
\end{proof}

\medskip\noindent\emph{Claim $3$:}
Suppose that there exist $h$-irreducible submodules $W_1, W_2$ of $W$, such that
$|W_i|=p^{a_i}$ with $0<a_1\leq a_2$
and $W_1\cap W_2=0$. Then $V=U\oplus W_1\oplus W_2$, $p=2$, $(a_1,a_2)=1$, $0\leq a\leq 2\leq
a_1<a_2$, $d\geq5$, and $h=J_a\pl s_{a_1}\pl s_{a_2}$ for some Singer cycles $s_{a_1}, s_{a_2}$. These elements have exactly four cycles and arise as examples in lines~\ref{line1T6} (if $a=0$), and~\ref{line7T6} (if $a>0$, with $v=e_1$) of Table~\ref{tab:new2}. 

\begin{proof}[Proof of Claim~$3$]
The subspace $V'=U\oplus W_1\oplus W_2\leq V$ is $g$-invariant, and
we have the following nonempty $g$-invariant subsets:
$U, (U\oplus W_i)\setminus U$ (for $i=1, 2$), and $V'\setminus
((U\oplus W_1)\cup (U\oplus W_2))$, of sizes $p^a$, $p^a(p^{a_i}-1)$
(for $i=1, 2$), and $p^a(p^{a_1}-1)(p^{a_2}-1)$.
Since $g$ has at
most four cycles, it follows that $V'=V$ and $\langle g\rangle$ acts transitively
on each of these subsets. 

Observe that if $(p^{a_1}-1,p^{a_2}-1)=\ell$, then $g$ induces at least $\ell$ cycles on $V\setminus((U\oplus W_1)\cup (U\oplus W_2))$. Thus
 $p^{a_1}-1$ and $p^{a_2}-1$ are coprime, and hence $p=2$ and $(a_1,a_2)=1$. Also transitivity
of $h$ on $W_i\setminus\{0\}$ implies that $h|_{W_i}$ is a Singer cycle $s_{a_i}$, and
by Claim~2, $h|_U=J_a$ and $a\leq2$. Since $(a_1,a_2)=1$ we have $a_1<a_2$ and since
$h|_{W_1}\ne 1$ (because $W_1\neq 0$ and the $1$-eigenspace of $h$ is contained in $U$), we have $a_1\geq2$. Thus $d=a+a_1+a_2\geq 5+a$, and the elements
are the examples with 4 cycles in lines~\ref{line1T6},~\ref{line7T6} of Table~\ref{tab:new2} 
(if $a\geq1$ we note that $g$ has a conjugate of the form $g=t_{e_1}h$).
\end{proof}

\subsection{Four cycles: proof of Theorem~\ref{thm:element4cycles}}

\begin{table}[t]
\begin{center}
\begin{tabular}{ccccll}
\hline
Line & $d$ & $\#$ cycles & $g$ & $|g|$ & Cycle lengths \\ \hline
 1& -- & $i+1$ & $s_{d}^i$ & $|h|$ & 1, and $i$ of length $\frac{p^{d}-1}{i}$ \tabfivelabel{line1T5} \\
 			&	 & 	($1 \le i \le 3$ and $i \mid p^d-1$)& & \\
 2 & $2$ & $3$ & $s_1\otimes J_2$ & $|h|$ & $1, p-1, p(p-1)$ \tabfivelabel{line2T5} \\
 3 & $1$ & $1$ & $t_{e_1}$ & $p|h|$ & $p$ \tabfivelabel{line3T5} \\
 4 & $\geq 2$ & $i+1$ & $t_{e_1}(J_1\pl s_{d-1}^i)$ & $p|h|$ & $p$, and $i$ of length $\frac{p(p^{d-1}-1)}{i}$ \tabfivelabel{line4T5}\\
 &				 & ($1 \le i \le 3$ and $i \mid p^{d-1}-1$)& & \\
\hline
\end{tabular}
\end{center}
\caption{At most $4$-cycles, arbitrary $p$}\label{tab:new1}
\end{table}

\begin{table}[t]
\begin{center}
\begin{tabular}{ccccll}
\hline
Line & $d$ & $\#$ cycles & $g$ & $|g|$ & Cycle lengths \\ \hline
 1 & $\geq 3$ & $4$ & $s_{a_1}\pl s_{a_2}$ & $|h|$ & $1, 2^{a_1}-1, 2^{a_2}-1$, \tabsixlabel{line1T6} \\
 & & & $(a_1,a_2)=1$, $d=a_1+ a_2$ & & $(2^{a_1}-1)(2^{a_2}-1)$ \\
 2 & $\geq3$ & $4$ & $J_1\pl s_{d-1}$ & $|h|$ & $1, 1, 2^{d-1}-1, 2^{d-1}-1$ \tabsixlabel{line2T6} \\
 3 & $\geq5$ & $4$ & $t_{e_1}(J_3\pl s_{d-3})$ & $|h|$ & $4, 4, 2^{d-1}-4, 2^{d-1}-4$ \tabsixlabel{line3T6} \\
 4 & $\geq 4$ & $4$ & $t_{e_1}(J_1\pl J_1\pl s_{d-2})$ & $p|h|$ & $2, 2$, $2^{d-1}-2, 2^{d-1}-2$ \tabsixlabel{line4T6}\\
 5 & $\geq 4$ & $i+1\in\{2,4\}$ & $t_{e_1}(J_2\pl s_{d-2}^i)$ & $p|h|$ & $4$, and $i$ of length $\frac{2^{d}-4}{i}$ \tabsixlabel{line5T6}\\
 &&$d$ even if $i=3$&&&\\
 6 & $\geq 5$ & $4$ & $t_{e_1}(J_2\pl J_1\pl s_{d-3})$ & $p|h|$ & $4, 4$, $2^{d-1}-4, 2^{d-1}-4$ \tabsixlabel{line6T6}\\
 7 & $\geq 6$ & $4$ & $t_{e_1}(J_a\pl s_{a_1}\pl s_{a_2})$ & $p|h|$ & $2^a, 2^a(2^{a_1}-1), 2^a(2^{a_2}-1)$, \tabsixlabel{line7T6} \\
 & & & $1\leq a\leq 2\leq a_1< a_2$ & & $2^a(2^{a_1}-1)(2^{a_2}-1)$ \\
 & & & $(a_1,a_2)=1$ & & \\
 8 & $\geq 6$ & $4$ & $t_{e_1}(J_4\pl s_{d-4})$ & $p|h|$ & $8, 8$, $2^{d-1}-8, 2^{d-1}-8$ \tabsixlabel{line8T6} \\
 \hline
\end{tabular}
\end{center}
\caption{At most $4$-cycles, other infinite families ($p=2$)}\label{tab:new2}
\end{table}

\begin{table}[t]
\begin{center}
\begin{tabular}{cccccll}
\hline
Line & $d$ & $p$ & $\#$ cycles & $g$ & $|g|$ & Cycle lengths \\ \hline
 1 & $2$ & $2,3$ & $p$ & $t_{e_1}$ & $p|h|$ & $p$ of length $p$ \tabsevenlabel{line8T7}\\
 2 & $2$ & $3$ & $3$ & $t_{e_1} J_2$ & $|h|$ & $3, 3, 3$ \tabsevenlabel{line3T7} \\
 3 & $3$ & $2$ & $2$ & $t_{e_1} J_3$ & $|h|$ & $4, 4$ \tabsevenlabel{line5T7} \\
 4 & $3$ & $2$ & $4$ & $t_{e_3}(J_2\pl J_1)$ & $|h|$ & $2, 2, 2, 2$ \tabsevenlabel{line7T7} \\
 5 & $3$ & $2$ & $4$ & $J_3$ & $|h|$ & $1,1,2,4$ \tabsevenlabel{line1T7} \\
 6 & $4$ & $2$ & $4$ & $t_{e_1}(J_3\pl J_1)$ & $|h|$ & $4, 4, 4, 4$ \tabsevenlabel{line6T7} \\
 7 & $4$ & $2$ & $4$ & $s_2\otimes J_2$ & $|h|$ & $1,3,6,6$ \tabsevenlabel{line2T7} \\
 8 & $5$ & $2$ & $4$ & $t_{e_1} J_5$ & $|h|$ & $8, 8, 8, 8$ \tabsevenlabel{line4T7} \\
 9 & $2$ & $2$ & $1$ & $t_{e_1}J_2$ & $p|h|$ & $4$ \tabsevenlabel{line9T7}\\
 10 & $3$ & $2$ & $4$ & $t_{e_1}$ & $p|h|$ & $2,2,2,2$ \tabsevenlabel{line10T7} \\
 11 & $3$ & $2$ & $2$ & $t_{e_1}(J_2\pl J_1)$ & $p|h|$ & $4,4$ \tabsevenlabel{line11T7}\\
 12 & $3$ & $3$ & $3$ & $t_{e_1}J_3$ & $p|h|$ & $9,9,9$ \tabsevenlabel{line12T7}\\
 13 & $4$ & $2$ & $4$ & $t_{e_1}(J_2\pl J_2)$ or $t_{e_1}(J_2\pl J_1\pl J_1)$ & $p|h|$ & $4,4,4,4$ \tabsevenlabel{line13T7}\\
 14 & $4$ & $2$ & $2$ & $t_{e_1}J_4$ & $p|h|$ & $8,8$ \tabsevenlabel{line14T7}\\
 15 & $5$ & $2$ & $4$ & $t_{e_1}(J_4\pl J_1)$ & $p|h|$ & $8, 8, 8, 8$ \tabsevenlabel{line15T7}\\
 \hline
\end{tabular}
\end{center}
\caption{At most $4$-cycles, sporadic cases}\label{tab:new3}
\end{table}

\begin{proof}[Proof of Theorem~$\ref{thm:element4cycles}$]
Let $g=t_vh\in \AGL_d(p)$ with at most four cycles in its action on $V$. Such an element $g$ must appear in Table~\ref{tab:neworders1},~\ref{tab:neworders2}, or~\ref{tab:neworders3}. We consider each possibility on a line-by-line basis. As before we use Notation~\ref{not}. Firstly, we suppose that $|g|=|h|$.

 If $g$ is as in line~\ref{l1T2} of Table~\ref{tab:neworders1}, 
 then we may assume that $g=h$ and we have the examples in line~\ref{line1T5} of Table~\ref{tab:new1}. 
 Similarly line~\ref{line1T4} of Table~\ref{tab:neworders3} 
 gives rise to line~\ref{line2T5} of Table~\ref{tab:new1} and line~\ref{line3T7} of Table~\ref{tab:new3}. 
 In line~\ref{line1T3} of Table~\ref{tab:neworders2} 
 we have $U= \langle e_1 \rangle$ and we may assume that $v \in U$. But if $v\ne 0$ then $|g|=p|h|$; so we may assume $g=h$ (recall that $p=3$ for this line). But then $g$ has $3$ cycles on $U$ and therefore has more than $4$ cycles in total by Claim~1.
 In line~\ref{line2T3} of Table~\ref{tab:neworders2} 
 $g=h=s_{d_1}\pl s_{d-d_1}$, but Claim~3 implies that $p=2$, whereas we have $p=3$. Suppose that $g$ is as in line~\ref{line3T3} of Table~\ref{tab:neworders2}. 
 Then by Lemma~\ref{sub-restrict}, $g=h=(s_1\otimes J_2)\oplus s_{d-2}^i$, since $h$ is fixed point free on $V$, and $V$ has a $g$-module decomposition $W_1\oplus W_2$, where $\dim W_1=2$ and $g$ has $3$ cycles on $W_1$; but there must also be at least one cycle on $W_2 \setminus \{ 0 \}$ and on $V \setminus (W_1 \cup W_2)$; so these elements do not provide examples. 
Next, for $g$ as in line~\ref{line6T3} of Table~\ref{tab:neworders2}, conjugating by a suitable $t_{v'}$, we may assume that $v\in\la e_1\ra$. So $g=t_{ae_1}(J_2\oplus s_{d-2})$ has cycle lengths on $U$ equal to $1,1,1,3,3$ (if $a=0$), or $3,3,3$ (if $a=\pm 1$), contradicting Claim 1.
In line~\ref{line2T4} of Table~\ref{tab:neworders3}, 
again by Lemma~\ref{sub-restrict}, we have $g=h$ and direct calculation shows that the cycle lengths are $24,24,24,8,1$. In line~\ref{line8T3} of Table~\ref{tab:neworders2}, 
 again we may assume $v=0$ and Claim~3 shows that $t=2$ and we have the examples in line~\ref{line2T6} of Table~\ref{tab:new2}. 
 In line~\ref{line9T3} of Table~\ref{tab:neworders2}, 
 $U=0$ so, by Lemma~\ref{sub-restrict}, $g=h$ and Claim~3 yields the examples in line~\ref{line1T6} of Table~\ref{tab:new2}. (Observe that Claim~3 immediately gives that the elements in lines~\ref{line14T3} and~\ref{line14primeT3} of Table~\ref{tab:neworders2} give rise to no examples.) In line~\ref{line10T3} of Table~\ref{tab:neworders2}, 
 we have $g=h$ by Lemma~\ref{sub-restrict}; in this case $s_2\otimes J_2$ has cycle lengths $1,3,6,6$ on a $4$-dimensional subspace $W_1$ and so the only examples occur when $d=4$; see line~\ref{line2T7} of Table~\ref{tab:new3}. 
 In line~\ref{line11T3} of Table~\ref{tab:neworders2}, 
 by conjugating by a suitable $t_{v'}$ we may assume that $v=0$ or $v=e_1$. Direct calculation shows that (within $U$) these two cases give cycle lengths $1,1,2,4$ and $4,4$ respectively. Clearly the first case cannot occur (since $d\ge 5$). In the second case, Claim~1 implies $g=t_{e_1}(J_3\pl s_{d-3})$ as in line~\ref{line3T6} of Table~\ref{tab:new2}.
 Similarly in line~\ref{line12T3} of Table~\ref{tab:neworders2}, 
 $v=0$ or $v=e_1$, but in the latter case we have $|g|=2|h|$. Thus $v=0$, but then $g$ has cycle lengths $1,1,2$ on $U$ contradicting Claim~1. 
In line~\ref{line4T4} of Table~\ref{tab:neworders3} 
we have $h=J_d$, and conjugating by a suitable $t_{v'}$, we may assume that $v\in\la e_1\ra$. Recalling that 
we have $|g|=|h|$, we may assume $g=J_2$ (line~\ref{line2T5} of Table~\ref{tab:new1}), 
 $J_3$ (line~\ref{line1T7} of Table~\ref{tab:new3}), 
 $t_{e_1}J_3$ (line~\ref{line5T7} of Table~\ref{tab:new3}), 
 $J_4$ (cycle lengths $1,1,2,4,4,4$), $J_5$ (cycle lengths 1,1,2,4,4,4,8,8) or $t_{e_1}J_5$ (line~\ref{line4T7} of Table~\ref{tab:new3}). 
 In line~\ref{line9T4} of Table~\ref{tab:neworders3}, 
 $h=J_2\pl J_1$ and conjugating by a suitable $t_{v'}$ we may assume that $v \in \langle e_1,e_3 \rangle$ and there are four possibilities for $v$. A computation shows that only the choices $v=0$ and $v=e_3$ give $|g|=|h|$; now another direct calculation shows that we only have an example when $v=e_3$; see line~\ref{line7T7} of Table~\ref{tab:new3}. 
 In line~\ref{line11T4} of Table ~\ref{tab:neworders3}, Claim~1 implies $g=t_{e_1}(J_3\pl J_1)$ as in line~\ref{line6T7} of Table~\ref{tab:new3}.
 This completes the analysis of the case $|g|=|h|$.
%
%

Henceforth, we shall assume that $|g|=p|h|$.
 First suppose that $g$ is as in line~\ref{l2T2} of Table~\ref{tab:neworders1}. 
 Then either $d=2$ and $p=2,3$, which gives the examples in line~\ref{line8T7} of Table~\ref{tab:new3}; or $(p,d,i)=(2,3,3)$, as in line~\ref{line10T7} of Table~\ref{tab:new3} (these are the possibilities that have $s_{d-1}^i=1$); 
 or else we have the examples in line~\ref{line4T5} of Table~\ref{tab:new1} (when $s_{d-1}^i\neq 1$). 
If $g$ is as in line 3 of Table~\ref{tab:neworders1} then $g=t_{e_1}$ as in line 3 of Table~\ref{tab:new1}.
 If $g$ is in line~\ref{line4T3} of Table~\ref{tab:neworders2} 
 then $p=3$ and either $g=t_{e_1}(J_1\pl J_1\pl s_{d-2})$, and $g$ has three cycles on $U$ contradicting Claim~1, or $g = t_{e_1}(J_1\pl s_{d_1} \pl s_{d_2})$ with $(d_1,d_2)=1$, and these examples do not occur by Claim~3 since $p=3$. Next, if $g$ is in line~\ref{line5T3} of Table~\ref{tab:neworders2}, 
 then a direct calculation shows that $g$ has cycle lengths $9,9,9$ on $U$ and so there are no examples by Claim~1. Next, suppose that $g$ is as in lines~\ref{line15T3},~\ref{line16T3} of Table~\ref{tab:neworders2}. Using the notation in Table~\ref{tab:neworders2}, we have either $d_1=1$ or $d_1\ge 2$. In the former case, $g=t_{e_1}(J_1\pl J_1\pl h'')$ or $g=t_{e_1}(J_2\pl J_1\pl h'')$ and Claim~1 implies that $h''$ is a Singer cycle; see lines~\ref{line4T6},~\ref{line6T6} of Table~\ref{tab:new2}. 
 In the latter case, we apply Claim~3 to deduce that $g$ must be as in line~\ref{line7T6} of Table~\ref{tab:new2}. 
 Now suppose that $g$ is as in line~\ref{line17T3} of Table~\ref{tab:neworders2} 
 so $g=t_{e_1}(J_4\pl h')$; if $h'= s_{d-4}^i$ then we have the examples in line~\ref{line15T7} of Table~\ref{tab:new3} (when $h'=1$) and line~\ref{line8T6} of Table~\ref{tab:new2} (when $h'\neq 1$. Here, observe that $i=1$ by Claim~1, and $h'$ cannot be as in line~\ref{line9T3} of Table~\ref{tab:neworders2} by Claim~3). If $g$ is as in lines~\ref{line6T4},~\ref{line7T4} of Table~\ref{tab:neworders3} 
 then $vh(2) \ne 0$, hence $v$ generates a cyclic $h$-submodule of order $2^2$ and we may assume that $v=e_1$. A direct calculation gives us the examples in line~\ref{line13T7} of Table~\ref{tab:new3}. 
 Direct calculation shows that lines~\ref{line5T4},~\ref{line3T4},~\ref{line10T4} of Table~\ref{tab:neworders3} 
 give rise to the examples in lines~\ref{line14T7},~\ref{line9T7},~\ref{line12T7},~\ref{line10T7} of Table~\ref{tab:new3} respectively. 
 Similarly line~\ref{line18T3} of Table~\ref{tab:neworders2} and line~\ref{line8T4} of Table~\ref{tab:neworders3} 
 give the examples in line~\ref{line5T6} of Table~\ref{tab:new2} and line~\ref{line11T7} of Table~\ref{tab:new3} respectively. 
\end{proof}

 \section{Maximal subgroups of $\GL_d(p)$ containing elements of large order}

Let $g=t_vh\in\AGL(V)$ have order at least $|V|/4=p^d/4$, so $g$ is as in one of the lines of Tables~$\ref{tab:neworders1}$,~$\ref{tab:neworders2}$, or~$\ref{tab:neworders3}$. In this section we determine which kinds of primitive subgroups of $\AGL(V)$ contain at least one such element. Each primitive subgroup of $\AGL(V)$ is a semidirect product $G=TH$ where $T$ is the group of translations of $V$ and $H\leq\GL(V)$ is irreducible on $V$. It is convenient to use Aschbacher's description in~\cite{Asch} of the maximal subgroups $H$ of $\GL(V)$ not containing $\SL(V)$ (as exploited, for example in~\cite{BamPen,GPPS}). Thus we consider this problem class by class, for maximal subgroups in the various Aschbacher classes $\C_2, \dots, \C_9$. We discover that subgroups in many Aschbacher classes seldom contain elements of sufficiently large order.
 First we consider Aschbacher class $\C_2$: here the subgroups are stabilizers $\GL_{d/r}(p)\wr\Sym(r)$ of decompositions $V=\oplus_{i=1}^rV_i$, for some divisor $r$ of $d$ with $r>1$.

\begin{lemma} \label{lem:onlyC2s}
Let $ d \ge 3$ and let $r$ be a divisor of $d$ with $r>1$. Let $G=TH$ be a subgroup of $\AGL_d(p)$ with $H$ in the Aschbacher class $\C_2$ of type $\GL_{d/r}(p) \wr \Sym(r)$, and suppose that $G$ contains an element $g=t_vh$ with $|g|\geq p^d/4$. Then $p\in \{2,3\}$. 

If $p=3$, then either $r=2<d$, or $d=r= 3$. Moreover, the image of $h$ in $H/ \GL_{d/r}(p)^r \cong \Sym(r)$ is non-trivial only when $d=r= 3$.

If $p=2$, then either $d/r\geq 3$, or $d=r\leq 5$, or $d=2r\leq 6$. For $d/r\geq 3$, the image of $h$ in $\Sym(r)$ is trivial, and
moreover, $4r^2-21r\leq d$.
 \end{lemma}
\begin{proof}
Since $|g|\geq p^d/4$, by Theorem~\ref{thm:orders}, the element $h$ is as in Tables~$\ref{tab:neworders1}$,~$\ref{tab:neworders2}$, 
or~$\ref{tab:neworders3}$. Moreover, $|h|\geq |g|/|t_v|\geq p^{d-1}/4$. In the proof of this lemma we repeatedly use both of these observations on $h$.

By an inspection of Table~\ref{tab:neworders1} it is clear that $h=s_d^i$ and $h=J_1\pl s_{d-1}^i$ (with $1\leq i\leq 3$) are not contained in a $\C_2$ subgroup when $d\geq 3$. Therefore, since $d\geq3$, $h$ is as in one of the lines of Table~$\ref{tab:neworders2}$, or~$\ref{tab:neworders3}$, and in particular, $p\in \{2,3\}$. 

\smallskip 

Assume that $p=3$. Observe that, for every even $d$, $h=J_1\pl s_{d_1}\pl s_{d_2}$ with $d_2=d_1+1=d/2$ (as in line~\ref{line4T3} of Table~\ref{tab:neworders2}) lies in $\GL_{d/2}(3)\wr \Sym(2)$. Now assume that $r\geq 3$: we show that only $r=d=3$ is possible. For $d\geq 8$, the descriptions of $h$ in Tables~\ref{tab:neworders2} or~\ref{tab:neworders3} and a case-by-case analysis immediately eliminates $\C_2$-subgroups not of type $\GL_{d/2}(3)\wr \Sym(2)$. For $3 \le d\le 7$, there are maximal $\C_2$-subgroups only when $r=d$, or when $(d,r)=(6,3)$. A direct calculation eliminates the possibility $\GL_2(3) \wr \Sym(3)$ (the maximal element order of $\GL_2(3)\wr \Sym(3)$ is $48$, and $|h|\geq 3^{6-1}/4> 48$, a contradiction). Thus $3\leq r=d\leq7$. Another direct calculation shows that $T\cdot (\GL_1(3)\wr\Sym(d))$ contains elements $g=t_vh$ of order at least $3^{d}/4$ only when $d= 3$.

It remains to show that the image of $h$ in $H/\GL_{d/2}(3)^2$ is trivial when $r=2$. Suppose that $h=(h_1,h_2)(1 2)$, with $h_1,h_2\in \GL_{d/2}(3)$. Observe that $h^2=(h_1h_2,h_2h_1)$ and that $(h_1h_2)^{h_1}=h_2h_1$. Therefore $|h|=2|h_1h_2|\leq 2\meo(\GL_{d/2}(3))\leq 2(3^{d/2}-1)$. Since $|h|$ has order at least $3^{d-1}/4$, we get $3^{d-1}/4\leq 2(3^{d/2}-1)$, which has a solution only for $d=4$. Finally a computation in $T\cdot (\GL_2(3)\wr \Sym(2))$ shows that there is no element $t_vh$ of order $\geq 3^4/4$ with $h$ having non-trivial image in $\Sym(2)$.

\smallskip

Assume that $p=2$. Write $m=d/r$. We consider separately the cases $m=1$ and $m=2$. 
From~\cite[Theorem~$2$]{Mass}, we see that \[\log(\meo(\Sym(x)))\leq \sqrt{x\log(x)}\left(1+\frac{\log(\log(x))-0.975}{2\log(x)}\right)\] for $x\geq 3$, where $\log$ indicates the natural logarithm. For simplicity denote by $f(x)$ the exponential of the function on the right hand side of this inequality. For $m=1$, we have $H=\Sym(d)$. Now $|h|\geq 2^{d-3}$ and hence 
$2^{d-3}\leq f(d)$. A computation shows that this inequality is satisfied only when $d\leq 9$. Now for these small values of $d$, by computing the exact value of $\meo(\Sym(d))$ we see with another computation that $d\leq 5$.

Now we consider $m=2$, that is, $H=\GL_2(2)\wr\Sym(d/2)$. As $|\GL_2(2)|=6$, we get $2^{d-3}\leq |h|\leq \meo(\GL_2(2)\wr\Sym(d/2))\leq 6f(d/2)$ and a computation shows that this happens only for $d\leq 8$. Now for these small values of $d$ we see with another explicit computation that $d\leq 6$.

For the rest of the proof we assume that $m\geq 3$. We start by showing that $h \in H$ has trivial image in $H/ \GL_m(2)^r$. We write $h= (h_1,h_2, \ldots, h_r) \sigma$ where $h_i \in \GL_m(2)$ and $\sigma \in \Sym(r)$. We argue by contradiction and we assume that $\sigma\neq 1$. Suppose that $\sigma$ has a cycle of length $\ell$. If $\ell =r$ then without loss of generality, we may assume that $\sigma = (1 2 \cdots r)$. Now an easy computation shows that $(12 \cdots r)(h_1, h_2, \ldots, h_r) = (h_2, h_3,\ldots, h_r,h_1 ) (12 \cdots r)$. It follows that
 \begin{align*}
 h^r &= (h_1,h_2, \ldots, h_r) \sigma (h_1,h_2, \ldots, h_r) \sigma \cdots(h_1,h_2, \ldots, h_r) \sigma \\
 &= (h_1h_2\cdots h_r, \,h_2h_3 \cdots h_r h_1,\, \ldots,\, h_r h_1\cdots h_{r-1}).
 \end{align*}
 But
 \[ h_1 h_2 \cdots h_r = h_1 (h_2 \cdots h_r h_1) h_1^{-1} \]
 and similarly we see that all of the entries of $h^r$ above are conjugate. In particular, they have the same order and since $p=2$ we have $|h| \le r\meo(\GL_m(2))=r (2^{m}-1) < 2^{m+r-1}$. If $\ell=r\geq3$ then, since $m\geq3$, this is less than $2^{d-3}$. So the only possibility not eliminated yet in this case is $\ell=r=2$, and hence $\sigma$ is a transposition.

Next suppose that $ \ell < r$. Then $h\in (\GL_m(2) \wr \Sym(\ell))\times (\GL_m(2) \wr \Sym(r -\ell)),$ which is isomorphic to a subgroup of $(\GL_m(2)\wr \Sym(\ell)) \times \GL_{d-m\ell}(2)$. Using $\ell\leq 2^{\ell-1}$, the same calculation as above shows that $|h| \le \ell (2^{m}-1)(2^{rm-m\ell}-1) < 2^{rm +m +\ell -m \ell-1}$, and this is at most $2^{d-1}/4=2^{mr-3}$ when $m, \ell \ge 3$. Therefore all of the cycles of $\sigma$ must have length at most $2$. If $\sigma$ has at least two $2$-cycles then $h$ can be embedded in $(\GL_m(2)\wr \Sym(2)) \times (\GL_m(2)\wr \Sym(2)) \times \GL_{d-4m}(2)$ and the same argument shows that $|h| < 2^{d-3}$. It follows that $\sigma$ is a transposition.

When $\sigma$ is a transposition, up to reordering we may assume that $\sigma=(1 2)$. Now $h=(h_1,\ldots,h_r)(1 2)$ and $h^2=(h_1h_2,h_2h_1,h_3,\ldots,h_r)$. Since $h_1h_2$ and $h_2h_1$ are conjugate, we get
$|h|\leq 2(\meo(\GL_{m}(2)))^{r-1}=2(2^m-1)^{r-1}<2^{d-m+1}$. As $|h|\geq 2^{d-3}$, we obtain $d-3< d-m+1$, which gives $m< 4$. Thus $m=3$. With this information we can now refine our computations. In fact, for $m=3$, the group $\GL_3(2)$ has exponent $84$ and hence $\GL_3(2)^{r}$ also has exponent $84$. Thus $|h|\leq 2\cdot 84=168$. As $|h|\geq 2^{d-3}=2^{3r-3}$, we have $168\geq 2^{3r-3}$, which is satisfied only for $r\leq 3$. For $r=3$, it can be easily checked with a computer that the elements of $(\GL_3(2) \wr \Sym(2)) \times \GL_3(2)$ have order at most $56$. As $56<64=2^{d-3}$, this case does not arise. For $r=2$, it is a computation to verify that the maximal order of an element $g=t_vh$ of the affine group $T\cdot (\GL_3(2)\wr\Sym(2))$, with $h=(h_1,h_2)(1 2)$, is $14$. As $14<16=2^{6-2}$, the case $r=2$ does not arise either.

It remains to prove that $4r^2-21r\leq d$. From the previous paragraphs, we have $h=h_1\pl \cdots \pl h_r$, with $h_1,\ldots,h_r\in \GL_{m}(2)$. Recall that $|h|=\lcm\{|h_i|\mid i\in \{1,\ldots,r\}\}\geq 2^{d-3}$. If $|h_i|,|h_j|,|h_k|\leq 2^{m-1}$ for some distinct indices $i$, $j$ and $k$, then $|h|\leq 2^{3(m-1)}(2^{m}-1)^{r-3}<2^{d-3}$, a contradiction. This shows at most two entries of $h$ have order $\leq 2^{m-1}$. Up to reordering we may assume that $|h_j|>2^{m-1}$, for every $j\geq 3$, and an inspection of Tables~$\ref{tab:neworders1}$,~$\ref{tab:neworders2}$, or~$\ref{tab:neworders3}$ reveals that $h_j$ is as in line~\ref{l1T2} of Table~\ref{tab:neworders1} with $i=1$, or as in line~\ref{line9T3} of Table~\ref{tab:neworders2}, for each $j\geq 3$. If $h_j=h_k=s_{m}$ for some distinct indices $j$ and $k$, then $\lcm(|h_j|,|h_k|)= 2^m-1$ and hence $|h|\leq (2^{m}-1)^{r-1}<2^{d-m}\leq 2^{d-3}$. This shows that there exists at most one index with $h_j$ as in line~\ref{l1T2} of Table~\ref{tab:neworders1}. Therefore, up to reordering, we may assume that $h_j$ is as in line~\ref{line9T3} of Table~\ref{tab:neworders2} for each $j\geq 4$. 

For $i\in \{4,\ldots,r\}$ write $h_i=s_{d_{i,1}}\pl \cdots \pl s_{d_{i,t_i}}$, with $d_{i,1},\ldots,d_{i,t_i}\geq 2$ pairwise coprime and $t_i\geq 2$. Suppose that $d_{i,j},d_{i',j'},d_{i'',j''}$ are even, for some $i,j,i',j',i'',j''$ with $i$, $i'$ and $i''$ pairwise distinct. Then $\gcd(|h_i|,|h_{i'}|,|h_{i''}|)\geq 3$ and hence, arguing as above, $|h|\leq (2^m-1)^r/3^2<2^{d-3}$. So, up to reordering, we may assume that $d_{i,j}$ is odd for every $i\geq 6$ and for every $1\leq j\leq t_i$. 

Repeating the argument in the previous paragraph, we see that (up to the usual reordering) $d_{i,j}\neq 3$ for every $i\geq 7$ and for every $1\leq j\leq t_i$. Now, if $d_{i,j}=d_{i',j'}$ for some distinct $i$ and $i'$ with $i,i'\geq 7$, then we have $\gcd(|h_i|,|h_j|)\geq 2^5-1=31$ and a computation shows that $|h|\leq (2^{m}-1)^r/31<2^{d-3}$, which is a contradiction. Therefore the numbers $d_{i,j}$, with $7\leq i\leq r$ and $1\leq j\leq t_i$, are pairwise coprime, odd and not equal to $3$. Since $t_i\geq 2$, we have at least $2(r-6)$ such integers in $\{1,\ldots,m\}$. Since the number of odd numbers greater than $3$ in $\{1,\ldots,m\}$ is $\leq (m-3)/2$, we get $2(r-6)\leq (m-3)/2$, which gives the desired result.
\end{proof}

\begin{remark}\label{rem2}{\rm The lower bound on $d$ (as as function of $r$) when $p=2$ given in Lemma~\ref{lem:onlyC2s} can be improved, as follows:
\begin{equation*}
2(r-5)\leq \frac{d/r}{\log(d/r)}\left(1+\frac{3}{2\log(d/r)}\right).
\end{equation*} 
This is essentially a consequence of the last paragraph of the proof of Lemma~\ref{lem:onlyC2s}, which shows that in the interval $\{1,\ldots,d/r\}$ there are at least $2(r-6)$ distinct pairwise coprime numbers greater than 3, odd, and coprime to 3. Therefore, there must be at least $2(r-5)$ distinct primes in $\{1,\ldots,d/r\}$, and so $2(r-5)\leq \pi(d/r)$, where as usual $\pi(x)$ is the function counting the number of primes $\leq x$. Now $\pi(x)\leq x/\log(x)(1+3/(2\log(x)))$ by~\cite[Theorem~$1$]{Schon}.

In fact, we will now show (assuming the truth of the extended Goldbach conjecture, explained below,) that this improved bound is close to having the correct order of magnitude. Let $\pi_2(n)$ represent the number of ordered pairs of primes $(p,q)$ with $p<q$ and $n=p+q$. Suppose that $d'$ is large and even and $r=\pi_2(d')$, then if $d=d'r$ there is a $g \in \GL_{d/r}(2) \wr \Sym(r)$ with $|g|\ge 2^d/4=2^{d-2}$. By the definition of $\pi_2$, there exist $(p_1, q_1), \ldots (p_r, q_r)$ pairs of primes with $p_i<q_i$ and $p_i+q_i=d'=d/r$. Clearly, the primes $\{p_i, q_i \mid 1 \le i \le r\}$ are all distinct, and hence the numbers in $\{2^{p_i}-1,2^{q_i}-1 \mid 1 \le i \le r\}$ are pairwise coprime (since $(2^a-1,2^b-1)=2^{(a,b)}-1$). Let $g=(h_1, \ldots, h_r)$, where $h_i=s_{p_i} \oplus s_{q_i}$. Then 
\[|g|=\prod_{i=1}^r(2^{p_i}-1)(2^{q_i}-1)=2^d\prod_{i=1}^r\left(1-\frac{1}{2^{p_i}}\right)\left(1-\frac{1}{2^{q_i}}\right).\]
Observe that $p_i>2$ since $d'$ is even. So this gives $|g|>2^d \varepsilon^2$, where 
\[\varepsilon=\prod_{i=3}^\infty \left( 1-\frac{1}{2^i}\right).\]
It is not hard to compute that $\varepsilon^2 \sim 0.59$ so that $|g| >2^{d-2}$.

Now, the extended Goldbach conjecture claims that for $n$ large and even, there is a constant $C$ (given in the conjecture) such that
\[\pi_2(n) \ge C\frac{n}{(\log(n))^2}.\]
When $r=\pi_2(d')=\pi_2(d/r)$, this gives \[r \ge C \frac{d/r}{(\log(d/r))^2}.\]
This shows, as claimed, that (assuming the extended Goldbach conjecture) there exist $d$ and $r$ for which some $g \in \GL_{d/r}(2)\wr\Sym(r)$ has $|g|>2^d/4$ and $d$ and $r$ come close (asymptotically) to meeting the improved bound given in the first paragraph of this remark.
%
}
\end{remark}

 \begin{lemma} \label{lem:onlyC4s}
 If $H$ is a maximal subgroup of $\GL_d(p)$ of type $\C_4$ containing an element $h$ of order at least $p^{d-1}/4$, then $d=6$, $p\in \{2,3\}$ and $H = \GL_2(p) \otimes \GL_3(p)$. Moreover $\GL_2(p) \otimes \GL_3(p)$ contains elements of order at least $p^{6}/4$ if and only if $p=2$.
 \end{lemma}
 \begin{proof}
 By~\cite[Table~$3.5$A and (4.4.10)]{KL}, $H= \GL_{d_1}(p) \otimes \GL_{d_2}(p)$ where $2 \le d_1 < d_2$ and $d=d_1d_2$. It follows that
 \[
|h|\leq \meo(H) \le (p^{d_1}-1)(p^{d_2}-1) < p^{d_1+d_2}
\]
 and the last quantity is greater than $p^{d-1}/4$ if and only if $(d_1,d_2)=(2,3)$, or $p\in \{2,3\}$ and $(d_1,d_2)=(2,4)$. 

For $(d_1,d_2)=(2,4)$ and $p\in \{2,3\}$, a direct computation shows that $\meo(H)\leq p^{7}/4$ (in fact, $\meo(H)=30$ when $p=2$, and $240$ when $p=3$). Hence we may assume that $(d_1,d_2)=(2,3)$; that is, $H=\GL_2(p)\otimes \GL_3(p)$, and $d=6$. 

Assume $p\geq 5$ and write $h=h_1\otimes h_2$ with $h_1\in \GL_2(p)$ and $h_2\in \GL_3(p)$. If $|h_1|\leq (p^2-1)/4$ or $|h_2|\leq (p^3-1)/4$, then $|h|\leq |h_1||h_2|\leq (p^2-1)(p^3-1)/4<p^{5}/4$, a contradiction. Thus we may assume that $|h_1|>(p^2-1)/4$ and $|h_2|> (p^3-1)/4$. Since $|h_1|$ and $|h_2|$ are both integers and since $p\geq 5$, we must have $|h_1|\geq p^2/4$ and $|h_2|\geq p^3/4$. From Tables~\ref{tab:neworders1}, \ref{tab:neworders2} and~\ref{tab:neworders3} we have $h_1=s_2^i$ and $h_2=s_3^j$ (with $1\leq i,j\leq 3$), and a quick computation gives $|h|=|h_1\otimes h_2|<p^5/4$, which is a contradiction.

Finally, two straightforward computations show that $\meo(\GL_2(2)\otimes \GL_3(2))=21>2^6/4$ and $\meo(\GL_2(3)\otimes \GL_3(3))=104<3^6/4$.
 \end{proof}

 \begin{lemma} \label{lem:noC6s}
 If $H$ is a maximal subgroup of $\GL_d(p)$ of type $\C_6$, then $p\geq5$.
 \end{lemma}
 \begin{proof}
 When $p=2$, we note that there are no $\C_6$-subgroups of $\GL_d(2)$. For the conditions in Table~3.5.A of~\cite{KL} would require that $d=r^m$ for some prime $r \ne 2$ and that $p \equiv 1\pmod{r}$, which is not possible when $p=2$. If $p=3$ then the conditions in Table~3.5.A of~\cite{KL} imply that either $r=2$, $d=2^m$ and $p\equiv 1\pmod 4$, or $r\geq5$ and $p\equiv 1\pmod r$. Clearly, neither condition holds.
 \end{proof}

\begin{lemma} \label{lem:noC7s}
 If $H$ is a maximal subgroup of $\GL_d(p)$ of type $\C_7$, where $p=2,3$, then $H$ does not contain an element of order at least $p^{d-1}/4$.
 \end{lemma}
 \begin{proof}
 By~\cite[(4.7.6)]{KL}, $H= \GL_{m}(p) \wr \Sym(t)$ where $d=m^t$, $t\geq2$ and $m \ge 3$. 
 If $p=2$, it follows that
 \[
 \meo(H) \le (2^{m}-1)^t \meo(\Sym(t)) < 2^{mt+t}
 \]
 since $\meo(\Sym(t)) \le 2^{t}$ for all $t$ (see, for instance,~\cite[Theorem~$2$]{Mass}).
 The last quantity is at most $2^{d-3}$ if and only if $mt+t \le m^t -3$. It is easily verified that this is the case unless $(m,t) = (3,2)$. But a direct computation for $(m,t)=(3,2)$ shows that $\meo(H)=28$, which is less than $2^{6}$. A similar calculation shows that there are no examples when $p=3$.
 \end{proof}

 \begin{lemma} \label{lem:onlyC8s}
 If $d\ge 3$ and $H$ is a maximal subgroup of $\GL_d(p)$ of type $\C_8$ and contains an element $h$ as in Tables~$\ref{tab:neworders1}$,~$\ref{tab:neworders2}$, or~$\ref{tab:neworders3}$, then $H =\CSp_4(2)$, $\CSp_4(3)$,
 $\CSp_6(2)$ or $\GO_4^{+}(3)$ or $\GO_3(3)$.
 \end{lemma}
 \begin{proof}
Note that $|h|\geq p^{d-1}/4$. First observe that since $p$ is prime, there are no $\C_8$-subgroups of unitary type. Now suppose that $H$ is of symplectic type. In particular, $d$ is even. By~\cite[Lemma~2.10]{GMPSorders}, we have
 \[p^{d-1}/4\leq |h|\leq \meo(H) \le p^{d/2+1};\]
 we seek conditions on $p$ and $d$ for which $p^{d/2+1} \geq p^{d-1}/4$ or equivalently $4 p^{d/2+1} \geq p^{d-1}$.

Assume that $p \ge 5$. We have
 \[4 p^{d/2+1} < p^{d/2+2} \]
 and $p^{d/2+2} \le p^{d-1}$ if and only if $d \ge 6$. Thus $d=4$. By Tables~\ref{tab:neworders1},~\ref{tab:neworders2},~\ref{tab:neworders3} either $h=s_{4}^{i}$ or $J_1\pl s_{3}^{i}$, with $1\leq i\leq 3$. In the first case, $h$ acts irreducibly on $V$ and hence lies in a maximal torus of $\CSp_4(p)$ of order $(p-1)(p^2+1)$. Thus $(p^4-1)/3\leq |h|\leq (p-1)(p^2+1)$, which is easily seen to be false. In the second case, $h$ acts irreducibly on a $3$-dimensional subspace of $V$, however $\CSp_4(p)$ does not contain such elements.

 Assume that $p=3$. Then
 \[4 p^{d/2+1} < p^{d/2+3} \]
 and $p^{d/2+3} \le p^{d-1}$ if and only if $d \ge 8$. A direct calculation shows that $\meo(\CSp_6(3))=56 < 3^5/4$ and $\meo(\CSp_4(3))=24 \ge 3^{4}/4$, in fact $\CSp_4(3)$ is one of the groups in the statement of this lemma.

 Assume that $p=2$. Then $p^{d/2+1} < p^{d-1}/4$ if and only if $d/2+1 < d-3$ if and only if $d>8$. Direct calculation yields that $\meo(\CSp_8(2))=30 < 2^5$, but $\meo(\CSp_6(2))=15\ge 2^{3}$ and 
 $\meo(\CSp_4(2))=6 > 2^{3}/4$.
 So if $H$ is of symplectic type then all of the examples are listed in the Lemma.

 Now suppose $H$ is of orthogonal type, that is, $H=\GO_d^\varepsilon(p)Z$ where $Z$ is the subgroup of $\GL_d(p)$ of scalar matrices. Observe that by~\cite[Table~$3.5$A, Column~IV]{KL}, $p$ is odd because $H$ is maximal. 

Assume that $p \ge 5$. By Tables~\ref{tab:neworders1},~\ref{tab:neworders2},~\ref{tab:neworders3} since $d\ge 3$ we have two possibilities for $h$: either $h=s_d^i$ or $h=J_1\pl s_{d-1}^i$, with $1\leq i\leq 3$. In the first case, $h$ acts irreducibly on $V$ and hence (by considering the structure of the maximal tori of $H$) $d$ is even and $h$ lies in a maximal torus of order $(p-1)(p^{d/2}+1)$. Thus $(p^d-1)/3\leq |h|\leq (p-1)(p^{d/2}+1)$, which is easily seen to be false for every $d\geq 3$. In the second case, $h$ acts irreducibly on a subspace of $V$ of dimension $d-1$ and fixes a non-zero vector of $V$. By considering the structure of the maximal tori of $H$, we get that $d$ is odd and that $h$ lies in a maximal torus of order $\leq (p-1)(p^{(d-1)/2}+1)$. Thus $(p^{d-1}-1)/3\leq |h|\leq (p-1)(p^{(d-1)/2}+1)$, which is easily seen to be false for every odd $d\geq 5$. Therefore $d=3$ and $H=\GO_3(p)Z$. A computation shows that the matrix $h=J_1\pl s_2^i$ lies in $\GO_3(p)Z$ only if $h$ lies in $\GO_3(p)$. Therefore, $h$ has order at most $p+1$. Thus $(p^2-1)/3\leq |h|\leq p+1$, a contradiction.

Assume that $p=3$. Now from~\cite[Corollary~2.12]{GMPSorders} we see that $\meo(\GO_d^{\varepsilon}(3)) \le 3^{d/2+1}$; a direct calculation shows that this is less than $3^{d-1}/4$ unless $d\le 6$. Now it is straightforward to check that the only groups of orthogonal type containing elements in Tables~\ref{tab:neworders1},~\ref{tab:neworders2},~\ref{tab:neworders3} are those listed in the lemma.
\end{proof}

\begin{lemma} \label{lem:onlyC9s}
Let $d\geq 3$, let $p\in \{2,3\}$ and let $H $ be a subgroup of type $\C_9$ with $H$ maximal in $\SL_d(p)$ or maximal in $\GL_d(p)$. If $H$ contains an element $h$ with $|h| \ge p^{d-1}/4$, then $p=2$ and $(H,d) = (\Alt(6),3)$ or $(\Alt(7),4)$, or $p=3$ and $(H,d)=(2.M_{11},5)$.
\end{lemma}
\begin{proof}
We use the ``bar notation'' to denote the natural projection of $\GL_d(p)$ onto $\PGL_d(p)$. Observe that $\overline{H}$ is an almost simple group containing an element of order $\geq p^{d-1}/(4(p-1))$. Let $\overline{H}_0$ be the socle of $\overline{H}$. By~\cite[Corollary~4.3]{Liebeck11}, if $H \in \C_9$ then either
\begin{itemize}
\item[(i)] $|\overline{H}|<p^{2d+4}$; or
\item[(ii)] $d=(m-1)m/2$ and $\overline{H}_0=\PSL_m(p)$; or
\item[(iii)] $d=27$, $16$ or $11$ and $\overline{H}_0= E_6(p)$, $\POmega^{+}_{10}(p)$ or $M_{24}$ respectively.
\end{itemize}
Note that the alternating groups $\Alt(n)$ acting on their deleted permutation modules of dimensions $n-1$ or $n-2$ do not arise since such groups are contained in an orthogonal or symplectic group and so do not give rise to maximal $\C_9$-subgroups~\cite[p. 440-441]{Liebeck11}. 
Suppose that~(iii) holds. It is easy to check with~\cite{ATLAS} that for $\overline{H_0}=E_6(2)$, $\POmega_{10}^+(2)$ and $M_{24}$, the group $\Aut(\overline{H}_0)$ does not contain an element of order at least $2^{d-1}/4$. If $p=3$, then using~\cite[Table~A.7]{KS}, we have 
\[\meo(\Aut(E_6(p)))\leq |\Out(E_6(p))|\meo(E_6(p))=2(3,p-1)\frac{(p+1)(p^5-1)}{(3,p-1)}<\frac{p^{27-1}}{(4(p-1))}.\] 
Similarly, using ~\cite[Table~A.5]{KS}, if $p=3$, then 
\[\meo(\Aut(\mathrm{P}\Omega_{10}^+(p)))\leq |\Out(\mathrm{P}\Omega_{10}^+(p))|\meo(\mathrm{P}\Omega_{10}^+(p))\leq 2(p-1,4)\frac{(p^4+1)(p+1)}{(p-1,4)}<\frac{p^{16-1}}{4(p-1)}.\]

Suppose that~(ii) holds. Observe that $m\geq 3$ because $\overline{H}_0$ must be simple. Moreover, for $m=3$, we have $d=3=m$ and hence $\SL_d(p)\leq H$, which is a contradiction. For $m=4$, we have $d=6$ and the embedding of $\mathrm{PSL}_4(p)$ into $\PSL_d(p)$ described in~\cite[Section~$4$]{Liebeck11} is determined by the action of $\PSL_4(p)$ on the wedge product $\wedge^2W$, where $W$ is the natural $4$-dimensional module of $\mathrm{PSL}_4(p)$. However, this is exactly the embedding that determines the isomorphism $\PSL_4(p)\cong\POmega_6^+(p)$. Therefore, since we are assuming that $H\in \C_9$, we must also have $m\neq 4$. Thus $m\geq 5$. From~\cite[Table~3]{GMPSorders}, for $(m,p)\neq (3,2)$, we have $\meo(\overline{H}) \le \meo(\Aut(\overline{H}_0))=(p^{m}-1)/(p-1)$. Now a computation shows that
the inequality $(p^m-1)/(p-1)\geq p^{d-1}/(4(p-1))$ is never satisfied.

Now suppose that~(i) holds. Assume that $d\geq 10$ and $p=2$. In particular, $\overline{H}=H$ and $H$ contains an element of order $\geq 2^{d-1}/4=2^{d-3}$. We claim that there are no examples here. For suppose that $H_0=\PSL_m(q)$, for some $m$ and for some prime power $q$. From~\cite[Table~3]{GMPSorders} we have $\meo(H) \le (q^{m}-1)/(q-1)$ or $(m,q)\in\{(2,4),(3,2)\}$. If $\meo(H) \le (q^{m}-1)/(q-1)$ then $2^{d-3}\le (q^{m}-1)/(q-1)$, while $2^{2d+4}> |H| > \frac{1}{2(m,q-1)}q^{m^2-1}$ (see~\cite[Proposition~3.9(i)]{Bur2} for example). A direct calculation shows that these bounds cannot both hold when $d \ge 10$. If $(m,q)=(2,4)$ or $(3,2)$ and $d \ge 10$ then it is clear that $H$ cannot contain an element of order at least $2^{d-3}$. Similarly we take each possible simple group of Lie type in turn and direct calculation shows that the analogous bounds cannot hold when $d \ge 10$. We use the bounds on $\meo(H)$ from~\cite[Table~5]{GMPSorders}. For example if $H_0= {^2}F_4(q)$, where $q=2^{f}$, then we have $q^{26}/2 < |H| < 2^{2d+4}$ but $\meo(H) \le 16f(q^2 \sqrt{2q^{3}} +q+ \sqrt{2q}+1)$ and so $16f(q^2 \sqrt{2q^{3}} +q+ \sqrt{2q}+1) \ge 2^{d-3}$. If $d \ge 10$ then these bounds can only hold when $d=11$ and $q=2$ but it is straightforward to check in~\cite{ATLAS} that in this case $\meo(H) \le 20 < 2^{d-3}$. As a final example, if $H_0 = {^2}B_2(q)$
 where $q =2^{f}\ge 8$, then we have $\meo(H) \le f(q+\sqrt{2q}+1)$ so we have the bounds $f(q+\sqrt{2q}+1) \ge 2^{d-3}$, $q^{5}/2 \le |H| < 2^{2d+4}$, and $d \ge 10$. Direct calculation finds that these bounds are only satisfied when $f=5$ and $d=10$, but then we can check in \texttt{magma} that $\meo(\Aut(^2B_2(2^{5})))=41 < 2^{10-3}$.
 If $H_0=\Alt(m)$ ($m \ge 5$), then we have $2^{d-3} \le \meo(H) < e^{3/2\sqrt{m\log(m)}}$ by~\cite{La1,Mass}. We also have $m!/2 < |H| < 2^{2d+4}$ and a direct calculation shows that these bounds can only hold if $d \le 16$. But if $10 \le d \le 16$ then the bounds imply $m \le 14$ and we can obtain a much sharper upper bound on $\meo(H)$ by calculating the explicit value of $\meo(\Aut(\Alt(m)))$ in \texttt{magma}. Further direct calculation then shows that these stronger bounds cannot hold when $d \ge 10$.
 We note that if $H$ is a sporadic group, then~\cite{ATLAS} tells us that we have $2^{d-3} \ge \meo(H)$ for $d \ge 10$.

Assume that $d\geq 10$ and $p=3$. We carry out the same analysis as for $p=2$ and $d \ge 10$ and we see that no example arises.

Assume that $d \le 9$. Using the tables in Kleidman's thesis~\cite{Kthesis}, the only $\C_9$ subgroups in $\GL_d(2)$ with $d \le 9$ are $\Alt(6) \le \GL_3(2)$, $\Alt(7) \le \GL_4(2)$ and $\PGL_3(4).2 \le \GL_9(2)$. But $\meo(\Aut(\PSL_3(4)))=21 < 2^6$ and so we are left with the two examples in the lemma. The only $\C_9$ subgroups in $\PGL_d(3)$ with $d \le 9$ have $(\overline{H}_0,d) = (M_{11},5)$, $(\PSL_2(11) ,6)$, $(\PSL_3(3),6)$ and $(\PSL_3(9),9)$. Calculating $\meo(\Aut(\overline{H}_0))$ precisely in \texttt{magma} in each case yields that this is less than $3^{d-1}/8$ unless $(\overline{H}_0,d) = (M_{11},5)$ and this case is listed in the lemma.
\end{proof}

\begin{remark}\label{remAsch}{\rm
The reader may have noticed that the lemmas in this section consider the Aschbacher classes $\C_2$, $\C_4$, $\C_6$, $\C_7$, $\C_8$, and $\C_9$. Since $G_0=G\cap\GL(V)$ acts irreducibly on $V$, type $\C_1$ cannot arise. The elements in $\C_5$ are stabilizers of subfields of $\mathbb F_p$, however, since $|\mathbb F_p|$ is prime, there is no proper subfield and hence $\C_5$ is empty. The groups of type $\C_3$ will be considered in our proof of Theorem~\ref{thm:HA}, and will give rise to some examples. When $G_0$ is contained in a $\C_3$-subgroup, note that elements of $\C_3$ are stabilizers of extension fields of $\mathbb F_p$ of prime index, that is, subgroups isomorphic to $\GL_a(p^b)\rtimes C_b$ with $d=ab$ and $b$ prime. In particular, if $h$ lies in one of these groups we see that the dimensions of an indecomposable decomposition of $h^b$ can be grouped together so that the dimension of every group is a multiple of $b$. A tedious inspection of Tables~\ref{tab:neworders1}, \ref{tab:neworders2}, and~\ref{tab:neworders3} eliminates most of the cases.
}
\end{remark}

 \section{Proofs of Theorems~\ref{thm:HA} and~\ref{thm:HA4cycles}}

\begin{proof}[Proof of Theorem~$\ref{thm:HA}$]
Write $g=t_vh$, with $t_v\in T$ and $h\in G_0$, and observe that $g$ and $h$ are in Tables~$\ref{tab:neworders1}$,~$\ref{tab:neworders2}$, or~$\ref{tab:neworders3}$. In particular $|h|\geq p^{d-1}/4$. If $G_0 \ge \SL_d(p)$ then part (1) of the statement holds, so we assume that $G_0 \not \ge \SL_d(p)$.
We divide the proof into various cases.

\medskip

\noindent\textsc{Case $d=2$. } For $p \le 13$, we use \texttt{magma} to verify that the only examples occur in Theorem~\ref{thm:HA}. So we may assume that $ p\ge 17$; in particular, $h=s_2^i$, or $J_1\pl s_1^i$ or $s_1^i\otimes J_2$, for some $1\leq i\leq 3$. Let $Z=Z(\GL_2(p))$ and consider $\PGL_2(p) = \GL_2(p)/Z$. We note that in all three cases, $|hZ|\geq(p-1)/3$. Now $G_0Z/Z$ is a (not necessarily proper) subgroup of a group $M$, where either $M$ is a maximal subgroup of $\PGL_2(p)$ (not equal to $\PSL_2(p)$) or $M$ is a maximal subgroup of $\PSL_2(p)$. The maximal subgroups of $\PGL_2(p)$ for $p$ odd are described in~\cite[Corollary~$2.3$]{King} (we use the terminology introduced in~\cite[Section~$2$]{King}), and the maximal subgroups of $\PSL_2(p)$ were determined by Dickson (see~\cite[Chapter~$3$, Section~$6$]{Suzuki}). Thus $G_0Z/Z$ is contained in one of the following groups $M$:

\begin{enumerate}
\item[(i)] a dihedral group of order $2(p-1)$ (setwise stabilizer of a pair of points),
\item[(ii)] a dihedral group of order $2(p+1)$ (setwise stabilizer of a pair of imaginary points),
\item[(iii)] a reducible subgroup of order $p(p-1)$ (point stabilizer),
\item[(iv)] $\Sym(4)$, $\Alt(4)$, or $\Alt(5)$.
\end{enumerate}

Since $G_0$ is irreducible, $M$ cannot be of type (iii). If $M$ is of type~(iv), then $(p-1)/3\leq |h|\leq\max\{\meo(\Sym(4)), \meo(\Alt(4)), \meo(\Alt(5))\}=5$, which is a contradiction since $p\geq 17$. If $M$ is of type~(i), then $G_0\leq \GL_1(p)\wr\Sym(2)$ and part~(3)(iii) of Theorem~\ref{thm:HA} holds. Suppose finally that $M$ is of type~(ii). Then $G_0\leq \GammaL_1(p^2)$ and, in order to conclude that part~(2) of Theorem~\ref{thm:HA} holds, we need to show that $[\GL_1(p^2):G_0\cap \GL_1(p^2)]\leq 3$. Observe that $|G_0|$ is coprime to $p$ and hence $h=s_2^i$ or $h=J_1\pl s_1^i$. In the first case $[\GL_1(p^2):G_0\cap \GL_1(p^2)]\leq [\GL_1(p^2):\GL_1(p^2)\cap \langle s_2^i\rangle]\leq i\leq 3$. In the second case $h=J_1\pl s_1^i$ fixes a non-zero vector and, as every non-identity element of $\GL_1(p^2)$ acts fixed point freely on $V\setminus\{0\}$, we have $\langle h\rangle\cap\GL_1(p^2)=1$. Since $|\GammaL_1(p^2):\GL_1(p^2)|=2$, we must have $|h|\leq 2$, contradicting the facts that $|h|\geq (p-1)/3$ and $p\geq 17$.

\medskip

\noindent\textsc{Case $d\geq 3$ and $p\geq 5$. } We have that $g=s_d^i$ or $t_{e_1}(J_1\pl s_{d-1}^i)$ (with $1\leq i\leq 3$) by Tables~\ref{tab:neworders1},~\ref{tab:neworders2},~\ref{tab:neworders3}. If $(d,p) \ne (6,5),(7,5)$, then~\cite[Lemma~2.1 and Theorem~2.2]{GM} imply that $h\in \GL_d(p)$ is either contained in a $\C_3$- or a $\C_8$-subgroup, or $h$ is contained in one of the $\C_9$-subgroups listed in~\cite[Table~1]{GM}. Analysing the possibilities in~\cite[Table~1]{GM}, we see that either $d=9$ and $G_0$ normalises $\SL_3(p^{2})$, or $G_0$ is contained in a $\C_3$-subgroup, or $G_0$ is contained in a $\C_8$-subgroup. In the third case, Lemma~\ref{lem:onlyC8s} shows that none of these subgroups contain elements of the required order. 

Suppose next that $d=9$ and that $G_0$ normalises $\SL_3(p^{2})$. This possibility is eliminated since the image of $G_0$ in $\PGL_d(p)$ is almost simple and hence $\meo(G_0)\le (p-1) \meo(\Aut(\PSL_3(p^2))) = (p^{6}-1)/(p+1)$, which is less than $p^{8}/4$. If $(d,p)=(6,5),(7,5)$ then we can check in \texttt{magma} that $G_0$ must be contained in a $\C_3$ subgroup in this case as well.

For $d\geq 3$, only the elements of the form $s_d^i$ are contained in $\C_3$ subgroups and (in this case) we can use~\cite{GPPS} and~\cite{BamPen} to show that the only possibility for $G_0$ is either to be as in~(1), or as in~(2) (here, the condition $[\GL_1(p^d):G_0\cap \GL_1(p^d)]\leq 3$ follows from $1\leq i\leq 3$). 

\medskip

\noindent\textsc{Case $3\leq d\leq 8$ and $p=2$, or $3\leq d\leq 7$ and $p=3$. } Here we can check the primitive groups of affine type in the libraries stored in \texttt{magma}. We list the possibilities in Table~\ref{tab:HAgroups}.

\medskip
\noindent\textsc{Case $d\geq 8$ and $p=3$. }By Tables~\ref{tab:neworders1},~\ref{tab:neworders2},~\ref{tab:neworders3}, we may assume that either $h=J_1\pl s_{d/2}\pl s_{d/2-1}$, or some power of $h$ has order a primitive prime divisor of $3^{e}-1$ with $e > d/2$. Applying Aschbacher's theorem, we see that $G_0\le \GL_d(3)$ must be contained in a subgroup of type $\C_i$ for some $i=1, \ldots,9$. Since $d\geq 8$, Lemmas~\ref{lem:onlyC2s},~\ref{lem:onlyC4s},~\ref{lem:noC6s},~\ref{lem:noC7s},~\ref{lem:onlyC8s},~\ref{lem:onlyC9s}, and Remark~\ref{remAsch} imply that either $G$ is as in (3)(ii), or $G_0$ is contained in a $\C_3$-subgroup. In the latter case, Tables~\ref{tab:neworders1},~\ref{tab:neworders2},~\ref{tab:neworders3} imply that $g=s_d^i$. Now we can use~\cite{GPPS} and~\cite{BamPen} to show that $G$ is as in (1) or (2). 
 We note that all of these subgroups contain elements from Tables~\ref{tab:neworders1},~\ref{tab:neworders2},~\ref{tab:neworders3}.

\medskip
\noindent\textsc{Case $d\geq 9$ and $p=2$. } Again we apply Aschbacher's theorem together with Lemmas~\ref{lem:onlyC2s},~\ref{lem:onlyC4s},~\ref{lem:noC6s},~\ref{lem:noC7s},~\ref{lem:onlyC8s},~\ref{lem:onlyC9s}, and Remark~\ref{remAsch}. Since $d \ge 9$, we find that the only possibilities are that $G$ is as in~(3)(i), or $G_0$ is contained in a $\C_3$-subgroup. In the latter case, Tables~\ref{tab:neworders1},~\ref{tab:neworders2},~\ref{tab:neworders3} imply that $g=s_a\pl s_b$ with $(a,b)=2$ and $G_0 \le \GL_{d/2}(4):2 = \GammaL_{d/2}(4)$, or $g=s_d^i$. Using~\cite{BamPen} we find that $G$ is as in (1) or (2) when $g=s_d^i$. 

Next suppose then $g=s_a\pl s_b$ and $G_0\leq \GammaL_{d/2}(4)$. We claim that the only possibility is that $G_0$ contains $\SL_{d/2}(4)$ as in case (1). We argue by contradiction and we suppose that $G_0$ does not contain $\SL_{d/2}(4)$. Observe that since $g$ has odd order, we have $g\in G_0\cap \GL_{d/2}(4)$. Since $g\in G_0\cap \GL_{d/2}(4)$, the element $g$, when viewed as an element of $\GL_{d/2}(4)$, is of the form $s_{a/2}\pl s_{b/2}$ (and $(a/2,b/2)=1$). 
 Without loss of generality, suppose that $a/2 >b/2$. Let $\ell$ be the largest divisor of $4^{a/2}-1$ that is relatively prime to $4^{m}-1$ for every $1\leq m<a/2$. By~\cite[Lemma~$2.1$~$(c)$]{GM} and the comments that precede that lemma, we see that either $\ell>a+1$, or $a/2 \in \{3,6\}$. Observe that when $a/2=3$ we must have $(d,a,b)=(10,6,4)$ (because $(a,b)=2$ and $d\geq 9$), and when $a/2=6$ we must have $(d,a,b)\in \{(14,12,2),(22,12,10)\}$ (because $(a,b)=2$). 

Now we deal with the three possibilities $(d,a,b)\in \{(10,6,4),(14,12,4),(22,12,10)\}$. Here $g\in G_0\leq \GL_{d/2}(4)$ and some power of $g$ has order $a+1$, a primitive prime divisor of $4^{a/2}-1$. We can check easily in \texttt{magma} that if $G_0 \le \GL_5(4)$ is irreducible and contains $g=s_6\pl s_4$, then $G_0$ contains $\SL_{5}(4)$. So we may assume that $(d,a,b) \ne (10,6,4)$. Now an analysis (using~\cite{GPPS}) shows that if $(d,a,b) = (22,12,10), (14,12,2)$, then the only irreducible $G_0$ containing $g$ must contain $\SL_{d/2}(4)$ as in case (1) (the analysis in our two cases is straightforward since, in the notation of~\cite{GPPS}, we have $r=13$, $e=6$, $r=2e+1$, and $d=7$ or $11$, and so there are very few possibilities for $G_0$).

It remains to consider the case that $\ell>a+1$, where $\ell$ is the largest divisor of $4^{a/2}-1$ coprime to $4^m-1$ for every $1\leq m<a/2$. Now a power of $g$ has order $\ell$, and~\cite[Theorem~2.2]{GM} applied to this power of $g$ implies that the irreducible subgroup $G_0\cap \GL_{d/2}(4)$ of $\GL_{d/2}(4)$ 
 \begin{itemize}
 \item[(i)] contains $\SL_{d/2}(4)$ (but we are assuming this is not the case), or
 \item[(ii)] is contained in $\GU_{d/2}(2)$, $\GSp_{d/2}(4)$, or $\GO^{\epsilon}_{d/2}(4)$, or
 \item[(iii)] preserves an extension field structure (but this is not the case since $(a/2,b/2)=1$), or
 \item[(iv)] normalizes $\GL_{d/2}(2)$, or
 \item[(v)] normalizes one of the nine subgroups listed in~\cite[Table~1]{GM}.
 \end{itemize}
 In particular, $G_0\cap \GL_{d/2}(4)$ satisfies either~(ii),~(iv) or~(v). Using $d\geq 9$ and $(a/2,b/2)=1$, an immediate check of~\cite[Table~$1$]{GM} reveals that no example arises in our case.
 A calculation shows that $\GU_{d/2}(2)$, $\GSp_{d/2}(4)$, and $\GO^{\epsilon}_{d/2}(4)$ do not contain elements of order $|g| = (4^{a/2}-1)(4^{b/2}-1)/3$ (see~\cite{GMPSorders} for example) so $G_0\cap \GL_{d/2}(4)$ does not satisfy (ii) either. Similarly, if $G_0$ satisfies~(iv), then $G_0$ cannot contain elements of order as large as $|g|$. 

Thus we have shown in all cases that $G_0$ satisfies one of the conditions (1)--(4) in the statement of Theorem~\ref{thm:HA}. 
 \end{proof}

\begin{proof}[Proof of Theorem~$\ref{thm:HA4cycles}$]
Since $G$ contains an element $g=t_vh$ with at most four cycles on $V$, we have $|g|\geq p^d/4$ and hence $G$ and $G_0$ appear in Theorem~\ref{thm:HA}. The examples in (2) of Theorem~\ref{thm:HA} contain elements of the form $s_d^i$ (for $1\leq i\leq 3$), and these elements have at most four cycles; thus we have the examples in (1) of Theorem~\ref{thm:HA4cycles} with $r=d$. Now suppose that $G$ is as in (1) of Theorem~\ref{thm:HA}. When $d \le 8$ and $p=2$, or $d\le 7$ and $p=3$, or $d=2$ and $p \le 13$ we check in \texttt{magma} that the only examples appear in Theorem~\ref{thm:HA4cycles}. So we suppose that $d$ and $p$ do not satisfy these bounds.

 First suppose that $r=1$.
If $p=2$ then $G_0=\GL_{d}(2)$ as in (1) of Theorem~\ref{thm:HA4cycles}.
If $p=3$ then $G_0$ contains $\SL_{d}(3)$, which contains $s_d^{2}$ as in (1) of Theorem~\ref{thm:HA4cycles}.
If $p \ge 5$ then Tables~\ref{tab:new1},~\ref{tab:new2},~\ref{tab:new3} imply that $h=s_1 \otimes J_2$ (and $d=2$) or $h= s_d^i$ or $h = J_1 \pl s_{d-1}^i$ (for $i=1,2,3$).
Since $G_0$ contains $\SL_d(p)$ and $h$, and since $\det(h)$ has multiplicative order $(p-1)$, $(p-1)/2$ or $(p-1)/3$, it follows that $G_0$ contains $s_d^{i}$ as in (1) of Theorem~\ref{thm:HA4cycles}.

Now suppose that $r \ge 2$. The analysis in the proof of Theorem~\ref{thm:HA} implies that (under our restrictions on $d$ and $p$) if $g \in G$ then $g=s_d^{i}$, or $p=2$ with $g=s_{a}\pl s_b$, $(a,b)=2$ and (therefore) $r=2$. But if $p=r=2$ then $G_0$ contains $\SL_{d/2}(4)$, which contains $s_d^{i}$. Thus $G$ satisfies (1) of Theorem~\ref{thm:HA4cycles} in all cases of (1) of Theorem~\ref{thm:HA}. 

We verify using \texttt{magma} that the only groups in~(4) of Theorem~\ref{thm:HA} that admit a permutation with at most four cycles are those indicated with a ``y'' in Table~\ref{tab:HAgroups}.

Finally, suppose that $G$ and $G_0$ are as in~(3) of Theorem~\ref{thm:HA}. Assume that $p=2$. Thus $G_0\leq \GL_{d/r}(2)\wr\Sym(r)$ for some divisor $r$ of $d$ with $r>1$. If $d/r\leq 2$, then from Lemma~\ref{lem:onlyC2s} we have either $d=r\leq 5$ or $d=2r\leq 6$. It is a computation to show that in each of these cases $G$ contains an element with at most four cycles. So now suppose that $d/r\geq 3$. Now Lemma~\ref{lem:onlyC2s} implies that $h\in \GL_{d/r}(2)^r$. For $d>9$, with a direct inspection of Tables~\ref{tab:new1},~\ref{tab:new2},~\ref{tab:new3}, we see that $h$ is the sum of at most three indecomposable summands and hence $r\leq 3$. Moreover, a more careful inspection shows that in each case $h$ has an indecomposable summand acting irreducibly on a subspace of $V$ of dimension $\geq d/2$. Clearly this shows that $r=2$. (Observe that the case $r=2$ does arise from line~\ref{line7T6} of Table~\ref{tab:new2} with $a=1$, $a_1=d/2-1$ and $a_2=d/2$.) The cases $d\leq 8$ can be easily dealt with the help of a computer.

Assume that $p=3$. Thus $G_0\leq \GL_{d/r}(3)\wr\Sym(r)$ for some divisor $r$ of $d$ with $r>1$. If $d=r$, then Lemma~\ref{lem:onlyC2s} implies that $d=r\leq 3$. Now a computation shows that in each of these cases $G$ contains an element with at most four cycles. So now suppose that $d>r$, and Lemma~\ref{lem:onlyC2s} implies that $r=2$. A computation shows that $T\cdot (\GL_2(3)\wr\Sym(2))$ has no element with at most four cycles and hence we may assume that $d\neq 4$. Now Lemma~\ref{lem:onlyC2s} implies that $h\in \GL_{d/2}(3)^2$. Since $d\geq 6$, a direct inspection of Tables~\ref{tab:new1},~\ref{tab:new2},~\ref{tab:new3} implies that $h$ has an indecomposable summand acting irreducibly on a subspace of $V$ of dimension $\geq d-1$, which is clearly a contradiction.

Finally suppose that $p\geq 5$. In this case the element $J_1\oplus s_d$ is always contained in $\GL_1(p)\wr \Sym(2)$.
\end{proof}

\end{document}